\documentclass[11pt,a4paper]{article}

\usepackage[utf8]{inputenc}
\usepackage[T1]{fontenc}
\usepackage{microtype}
\usepackage[dvips]{graphicx}
\usepackage{xcolor}
\usepackage{times}

\usepackage[nonewpage]{imakeidx}
\makeindex[title=Index of notations, options= -s]

\usepackage[english]{babel}
\usepackage{amsmath,mathtools}
\usepackage{amsfonts}
\usepackage{mathrsfs}
\usepackage{amsthm}
\usepackage{subfigure}
\usepackage{amssymb}
\usepackage{dsfont}
\usepackage{bigints}
\usepackage{textgreek}
\usepackage{authblk}
\usepackage{titlesec}
\usepackage{enumitem}
\usepackage{url}
\usepackage{xfrac}   

\usepackage{bbm}

\usepackage[
breaklinks=true,
colorlinks=false,
bookmarks=true,bookmarksopenlevel=2]{hyperref}

\usepackage{geometry}
\newcommand{\m}[1]{
\mathbb{#1}
}
\newcommand{\q}{\mathcal}
\newcommand{\ds}{\displaystyle}

\newcommand{\e}{\varepsilon}

\newtheorem{theorem}{Theorem}[section]
\newtheorem{prop}[theorem]{Proposition}
\newtheorem{lem}[theorem]{Lemma}
\newtheorem{cor}[theorem]{Corollary}
\newtheorem{defi}[theorem]{Definition}

\theoremstyle{remark}
\newtheorem{rem}[theorem]{Remark}

\numberwithin{equation}{section}

\newcommand*\diff{\mathop{}\!\mathrm{d}}

\DeclareMathOperator\supp{supp}

\providecommand{\keywords}[1]
{
  \small	
  \textbf{{Keywords ---}} #1
  
  \medskip
}

\providecommand{\subjclass}[1]
{
  \small	
  \textbf{{2020 Mathematics Subject Classification ---}} #1
  
  \medskip
}

\DeclarePairedDelimiter\abs{\lvert}{\rvert}%
\DeclarePairedDelimiter\norm{\lVert}{\rVert}%
\newcommand{\tnorm}[1]{{\left\vert\kern-0.25ex\left\vert\kern-0.25ex\left\vert #1 
    \right\vert\kern-0.25ex\right\vert\kern-0.25ex\right\vert}}

\makeatletter
\let\oldabs\abs
\def\abs{\@ifstar{\oldabs}{\oldabs*}}
\let\oldnorm\norm
\def\norm{\@ifstar{\oldnorm}{\oldnorm*}}
\makeatother

\newcommand*{\myemail}[1]{%
    \normalsize\href{mailto:#1}{#1}
    }

\allowdisplaybreaks

\titleformat{\section}[block]{\centering \scshape \large}{\thesection.}{0.3\baselineskip}{}
\titlespacing{\section}{0pt}{*5}{*2}

\titleformat{\subsection}[block]{\bfseries}{\thesubsection.}{.5em}{}
\titlespacing{\subsection}{0pt}{*2.5}{*1}

\titleformat{\subsubsection}[runin]{\itshape}{\normalfont \thesubsubsection.}{.5em}{}[.]
\titlespacing{\subsubsection}{0pt}{*2.5}{0.5em}

\title{Asymptotic stability of 2-domain walls for the Landau-Lifshitz-Gilbert equation in a nanowire with Dzyaloshinskii-Moriya interaction}
\date{\vspace{-1em}}

\author[*]{Raphaël Côte}
\author[$\dag$]{Guillaume Ferriere}

\affil[*$\dag$]{Institut de Recherche Mathématique Avancée, UMR 7501, Université de Strasbourg, CNRS, Strasbourg, France}
\affil[*]{University of Strasbourg Institute for Advanced Study (USIAS)}
\affil[*]{\myemail{cote@math.unistra.fr}}
\affil[$\dag$]{\myemail{guillaume.ferriere@inria.fr}}

\begin{document}

\maketitle

\begin{abstract}
We consider a ferromagnetic nanowire, with an energy functional $E$ with  easy-axis in the direction $e_1$, and which  takes into account the  Dzya\-lo\-shin\-skii-Mo\-riya interaction. We consider configurations of the magnetization which are perturbations of two well separated domain wall, and study their evolution under the Landau-Lifshitz-Gilbert flow associated to $E$. 

Our main result is that, if the two walls have opposite speed, these configurations are asymptotically stable, up to gauges intrinsic to the invariances of the energy $E$. Our analysis builds on the framework developed in \cite{Cote_Ignat__stab_DW_LLG_DM}, taking advantage that it is amenable to space localisation.
\end{abstract}



\section{Introduction}

\subsection{A model for a ferromagnetic nanowire}

We model a ferromagnetic nanowire by a straight line $\m Re_1 \subset \m R^3$ (of infinite length) where 
\[ e_1= \begin{pmatrix} 
1 \\ 0 \\ 0
\end{pmatrix}, \quad
e_2= \begin{pmatrix} 
0 \\ 1 \\ 0
\end{pmatrix}, \quad
e_3 = \begin{pmatrix} 
0 \\ 0 \\ 1
\end{pmatrix}
\]
is the canonical basis of $\m R^3$. The magnetization $m = (m_1, m_2, m_3): \m R \to \m S^2$ of this nanowire takes its values into the unit sphere $\m S^2\subset \m R^3$,  and we associate to it the energy functional
\begin{align} \label{def:energy} \index{Functionals!$E$: energy}
E(m) = \frac{1}{2} \int_{\m R} |\partial_x m|^2 + 2\gamma \partial_x m \cdot (e_1\wedge m)+ (1-m_1^2) \, dx,
\end{align}
where  $x$ is the variable in direction $e_1$ of the nanowire and $\gamma \in \m R$ is a given constant with $|\gamma| <1$; it will be convenient to denote
\[ \index{Constants!$\Gamma$}
\Gamma \coloneqq \sqrt{1-\gamma^2}. \]
Here, $\cdot$ and $\wedge$ are the scalar and cross product in $\m R^3$. The term with $\gamma$ accounts for the Dzyaloshinskii-Moriya interaction. We refer to \cite{Cote_Ignat__stab_DW_LLG_DM} where this model was derived from the full 3D system by $\Gamma$-convergence in a special regime.

\bigskip

We are interested in the evolution of the magnetization under the Landau-Lifshitz-Gilbert flow associated to $E$, that is the equation:
\begin{equation} \tag{LLG} \label{eq:llg}
    \partial_t m = m \wedge H(m) - \alpha m \wedge ( m \wedge H(m) ),
\end{equation}
where now $m: I \times \m R \to \m S^2$ is the time dependent magnetization ($I$ is interval of time of $\m R$), $\alpha >0$ is the damping coefficient, and the magnetic field $H$ is given by
\begin{equation*}
    H(m) = - \delta E (m) + h(t) e_1.
\end{equation*}
 $\delta E (m)$ is the variation of the energy, which writes
\begin{equation*}
    \delta E (m) = - \partial_{xx}^2 m - 2 \gamma e_1 \wedge \partial_x m + m_2 e_2 + m_3 e_3.
\end{equation*}
(recall that $m_1^2+m_2^2+m_3^2=1$). Finally, the function $h: I \to \m R$ is the (given) intensity of an applied external field, which we stress that it depends solely on the time variable $t$, and is oriented on the axis $e_1$.

The \eqref{eq:llg} flow is equivariant under the following set of transformations:
\begin{enumerate}
\item[$\bullet$] translations in space $\tau_y m(x) = m(x-y)$ for $y\in \m R$, and
\item[$\bullet$] rotations $\ds R_\phi = \begin{pmatrix} 
1 & 0 & 0 \\
0 & \cos \phi & -\sin \phi \\
0 & \sin \phi & \cos \phi
\end{pmatrix}$ about the axis $e_1$ and angle $\phi\in \m R$.
\end{enumerate}
There is another symmetry: if $m$ solves \eqref{eq:llg} with parameter $\gamma$, then ${}^\sharp m(t,x):= m(t,-x)$ solves \eqref{eq:llg} with parameter $-\gamma$. We nonetheless leave this last symmetry aside (it does not play any role in modulation theory for example), and we are lead to define the group 
\[ \index{Spaces!$G$: gauge group}
G \coloneqq \m R \times \m R /2\pi \m Z \]
which naturally acts on function $w: \m R \to \m R^3$ as follows: if $g= (y, \phi) \in G$, $g.w := R_\phi \tau_y w$.
The action of $G$ preserves $\m S^2$ valued functions, and so acts on magnetizations; it also extends naturally to functions of space and time, for which it preserves solutions to \eqref{eq:llg}. Also, we endow $G$ with the natural quotient distance over $\m R^2$:
\[ \forall g = (y,\phi) \in G,\quad |g| \coloneqq |y| + \inf \{ |\phi+2k \pi|, k \in \m Z \}. \]

\bigskip

Our main object of interest here are (precessing) domain walls. These are explicit solutions studied in \cite{Cote_Ignat__stab_DW_LLG_DM} (to which we refer for further details): given $\sigma = (\sigma_1,\sigma_2) \in \{ \pm 1\}^2$ (we equivalently use the notation $\pm$ instead of $\pm 1$), denote
\begin{equation} \index{Functions!$w_*^\sigma$: domain wall} \index{Functions!$\theta_*$}
\label{formul1}
\forall x \in \m R, \quad w^\sigma_*(x)  \coloneqq \begin{pmatrix}  \cos ( \theta_*(\sigma_1 x))  \\ \sigma_2  \sin (\theta_*(\sigma_1 x))  \cos (\gamma x)  \\ \sigma_1 \sigma_2 \sin ( \theta_*(\sigma_1 x))  \sin (\gamma x) \end{pmatrix}  \quad
\textrm{with} \quad  \theta_*(x) \coloneqq 2\arctan (e^{- \Gamma x}),
\end{equation}
and
\begin{gather} \index{Functions!$g_*$, $y_*$, $\phi_*^\pm$: domain wall gauge}
g_*^\sigma := (\sigma_1 y_*, \phi_*^{\sigma_2}) \quad\text{where} \\
 \text{for } t \ge 0, \quad y_* (t)  \coloneqq - \frac{\alpha}{\Gamma} \int_0^t h(s) \diff s \quad \text{and} \quad \phi_*^\pm (t)  \coloneqq \left(-1 \pm \frac{\alpha \gamma}{\Gamma }\right) \int_0^t h(s) \diff s. 
 \label{eq_y_phi}
 \end{gather}
 Then 
 \begin{equation} \label{def:DM}
 (t,x) \mapsto g_*^{\sigma}(t). w^\sigma_*(x)
 \end{equation}
 is a solution to \eqref{eq:llg}, which we call a domain wall.

Recall that $w_*^\sigma$ are the only solutions, up to a gauge in $G$, to the static equation
\[ w \wedge \delta E (w) =0, \]
which satisfy the suitable limits at infinity $\lim_{\pm \infty} w = \pm \sigma_1 e_1$.
Moreover, they satisfy $\delta E (w_*^\sigma) = \beta_* w_*^\sigma$ where
\begin{equation} \label{eq:beta_*} \index{Functions!$\beta_*$}
    \beta_* \coloneqq 2 \Gamma^2 \sin^2 \theta_*.
\end{equation}
The case $\gamma=0$ (i.e., absence of DMI) corresponds to (in-plane) static domain walls where a rotation in $\theta_*$ of $180^\circ$ takes place along the nanowire axis $e_1$; these transitions are called Bloch walls (see e.g. \cite{CL06a, KK_these}). For future reference, we note that $\theta_*:\m R\to (0,\pi)$ solves the first order ODE 
\begin{equation}
\label{eq:theta*}
\partial_x \theta_* = -\Gamma \sin \theta_*, \quad \theta_*(-\infty)=\pi, \, \theta_*(+\infty)=0,
\end{equation}
and $w_*^\sigma$ satisfies the system of first order ODEs:
\begin{equation}
\label{ode_w*}
\partial_x w_*^\sigma= \sigma_1 \Gamma w_*^\sigma \wedge(e_1\wedge w_*^\sigma)-\gamma e_1\wedge w_*^\sigma.
\end{equation}
Note also that formulas \eqref{formul1}, \eqref{eq_y_phi} and \eqref{def:DM} make sense for all $\alpha \in \m R$; however the condition  $\alpha >0$ is the physically relevant one, and will be required in all the following analysis.


\subsection{Functional spaces and Cauchy problem}

We denote $H^s$ (and $L^p$) for the Sobolev space $H^s(\m R, \m R^3)$ with $s \ge 0$ (and the Lebesgue space $L^p(\m R, \m R^3)$ with $p\in [1, \infty]$, respectively). We also denote $\dot H^s$ for the homogeneous Sobolev space  whose seminorm is given via Fourier transform:
\begin{equation} 
\label{def:H^s} 
\| m \|_{\dot H^s}^2 := \frac{1}{2\pi} \int_\m R |\hat m (\xi)|^2 |\xi|^{2s} d\xi, \quad \text{where} \quad \hat m(\xi) = \int_{\m R} e^{-ix\xi} m(x)\, dx.
\end{equation}
(In particular, $\| m \|_{\dot H^2} = \| \partial_{xx} m \|_{L^2}$). We define for $s \ge 1$, the spaces
\begin{align} \index{Spaces!$\q H^1, \q H^s$: energy spaces}
\label{def:qH^s}  \q H^s &:= \{ m=(m_1, m_2, m_3) \in \mathscr C(\m R, \m S^2) :  \| m \|_{\q H^s} < +\infty \}\\ 
\nonumber &\textrm{with}\quad \| m \|_{\q H^s}:= \| m_2 \|_{L^2} + \| m_3 \|_{L^2} + \| m \|_{\dot H^s}.
\end{align}
The $\q H^s$ spaces are modelled on the usual Sobolev spaces $H^s$, but adapted to the geometry of the target manifold $\m S^2$ and to the energy functional $E$:  the main point is that $|m_1| \to 1$ at $\pm \infty$ so that $m_1 \notin L^2$. $\q H^1$ corresponds to the set of finite energy configurations $E(m) < +\infty$ in which case, the energy gradient $\delta E(m)\in H^{-1}$. Also if $m, \tilde m \in \q H^1$ with $m(\pm \infty)=\tilde m(\pm \infty)$, then $m - \tilde m \in H^1$. Moreover, if $w_*$ is a domain wall \eqref{formul1}, then every configuration $m\in \q H^1$ with $\|m-w_*\|_{\q H^1}$ small enough is actually close to $w_*$ in $H^1$ with Lipschitz bounds, i.e., $\|m-w_*\|_{H^1}\lesssim \|m-w_*\|_{\q H^1}$
(we refer to \cite{Cote_Ignat__stab_DW_LLG_DM} for details and proofs).

Note that all the derivatives of $w_*^\sigma$ of order $k\ge 1$ are exponentially localised, so that
$w_*^\sigma\in {\q H}^k$ for all $k\ge 1$.

We use the following well posedness result, quoted from \cite{Cote_Ignat__stab_DW_LLG_DM}: see section 4 there, and the reference therein for more comments.

\begin{theorem}[Local well-posedness in $\q H^s$] \label{th:lwp} Let $\alpha >0$, $\gamma \in (-1,1)$ and $h \in L^\infty([0,+\infty), \m R)$. Assume $s \ge 1$ and $m_0 \in \q H^s$. Then there exist a maximal time $T_+= T_+(m_0) \in (0, +\infty]$ and a unique solution $m \in\mathscr C([0, T_+), \q H^s)$ to \eqref{eq:llg} with initial data $m_0$. 

Moreover,
\begin{enumerate}
\item if $T_+ <+\infty$, then $\| m(t) \|_{\q H^1} \to +\infty$ as $t \uparrow T_+$;
\item for $T<T_+$ (with $T_+$ finite or infinite), the map $\tilde m_0\in \q H^s \to \tilde m\in \mathscr C([0,T], \q H^s)$ is continuous in a small $\q H^s$ neighbourhood of $m_0$ (for every initial data $\tilde m_0$ in that neighborhood, the maximal time of the corresponding solution $\tilde m$ satisfies $T_+(\tilde m_0) > T$);
\item if $s \ge 2$, one has the energy dissipation identity
: $t \mapsto E(m(t))$ is a locally Lipschitz function in $[0,T_+)$ (even $\mathscr C^1$ provided $h$ is continuous) and for all $t \in [0,T_+)$,
\begin{align}
\frac{d}{dt} E(m) = - \alpha \int ( |\delta E(m)|^2 - |m \cdot \delta E(m)|^2 )\, dx + \alpha h(t) \int (m \wedge e_1) \cdot (m \wedge \delta E(m)) \, dx. \label{eq:en_dissip2}
\end{align}
\end{enumerate}
\end{theorem}

\subsection{Statement of the main result}

In \cite{Cote_Ignat__stab_DW_LLG_DM}, the flow of \eqref{eq:llg} around the domains wall \eqref{def:DM} was studied: for small $H^1$ perturbation, and under a small applied field $h$ (in $L^\infty_t((0,+\infty))$), domains walls were proved to be (exponentially) asymptotically stable, up to a gauge. This work thus extended previous results in two directions: in the absence of Dzya\-lo\-shin\-skii-Mo\-riya interaction (case $\gamma=0$), precessing domain walls were reported in \cite{GRS10}, and their linear asymptotic stability was proved in \cite{GGRS11} (it however completely disregards the gauge involved); nonlinear stability was also checked numerically in \cite{GGRS11}. We can also mention earlier studies of stability for Bloch or Walker walls (which are travelling fronts, not precessing) under some variant of \eqref{eq:llg} (the DMI interaction is not taken into account in the energy $E$): we refer for example to \cite{CL06a,Jiz11,Car10,Car14,Tak11}.

\bigskip

In the present paper, we are further interested to study the dynamics of solutions to \eqref{eq:llg} in the presence of several domain walls. This question is not only academically relevant for the long time dynamics, but also motivated by application of this model to data storage:  domain walls encode information, and their stability property is important for the persistence of this storage over time.

The simplest case to tackle is the interaction of domain walls decoupling with time: in view of $y_*$, there is essentially one such configuration, where the speeds are opposite (the transition of these domains wall are centered at $y_*(t)$ and $-y_*(t)$ respectively, up to a fixed translation). This corresponds to studying the evolution of a perturbation of 
\begin{gather}
g_*^{+}(t). w_*^+ (x) + g_*^{-}(t). w_*^- (x) + e_1
\end{gather}
\index{Functions!$g_*^\pm$: gauge of the two separating domain walls}
\index{Functions!$w_*^\pm$: the two separating domain walls}
where given $\sigma_2,\sigma_2' \in \{ \pm 1 \}$, we let $w_*^+ \coloneqq w_*^{(1,\sigma_2)}$, $w_*^- \coloneqq w_*^{(-1,\sigma_2')}$, and $g_*^+ \coloneqq g_*^{(1,\sigma_2)}$ $g_*^- \coloneqq g_*^{(-1,\sigma_2')}$.

In the decomposition of a $\m S^2$ magnetisation around two decoupled domain walls, it is interesting to consider gauges in $G$ with large translation parameter, which motivate the notation, given $L>0$,
\[ \index{Spaces!$G_{>L}$, $G_{<-L}$: gauges with large translation parameter $y$}
G_{>L} \coloneqq \{ (y,\phi) \in G : y > L \}, \qquad G_{< - L} \coloneqq \{ (y,\phi) \in G : y < - L \}. \]

\begin{theorem} \label{th1} 
\index{Constants!$L_0$, $\delta_0$: in the main theorem} \index{Functions!$q$: interaction function}  \index{Functions!$\kappa$}
There exist $L_0,\delta_0 > 0$ and $C, \lambda > 0$ such that the following holds. Assume that $h$ satisfies 
\begin{equation} \label{eq:ass_h_small}
    \|h \|_{L^\infty ((0, \infty))} \le \delta_0,
\end{equation}
and that 
\begin{equation} \label{eq:ass_Y}
\quad \int_0^\infty \sqrt{q(2 y_*(t))} \diff t <+\infty \quad \text{where, for } r \in \m R, \quad   q(r) \coloneqq (1+|r|) e^{-\Gamma r}.
\end{equation}
Denote for $t \ge 0$,
\begin{align} \label{def:kappa}
\kappa(t) \coloneqq e^{-\Gamma y_*(t)} + \left( \int_0^t e^{-2\lambda(t-s)} q(2y_*(s)) ds \right)^{1/2}.
\end{align}
Let $m_0 \in \mathcal{H}^1$ such that that there exist $L \ge L_0$ and $\zeta^{+} \in G_{> L}$, $\zeta^{-} \in G_{< - L}$ with
\begin{equation} \label{ass:in_data} \index{Constants!$\delta$: size of the perturbation at time $0$; $\delta < \delta_0$}
    \delta \coloneqq  \norm{m_0 - \Bigl( \zeta^{+} . w_*^{+} + \zeta^{-} . w_*^{-}  + e_1 \Bigr) }_{H^1} \le \delta_0.
\end{equation} 
\index{Constants!$L$: (half) minimal distance between the domain walls, $L \ge L_0$}
Then the solution $m$ to \eqref{eq:llg} is global for forward times and there exist 2 gauges $g^{+}, g^{-} \in W^{1,\infty}(\m R_+,G)$, such that,
    \begin{gather} \label{est:conv_m}
 \forall t \ge 0, \quad \norm{m(t) - \Bigl( g^{+} . w_*^{+} + g^{-} . w_*^{-} + e_1 \Bigr) }_{H^1} \le C  (\delta + \sqrt{ q(2L)})  e^{-\lambda t}+ C \sqrt{ q(2L)})  \kappa(t).
    \end{gather}
    Moreover, there exist two gauges $g_\infty^{\pm} \in G$ such that
\begin{equation} \label{est:conv_g}
\forall t \ge 0, \quad    \sum_{\iota \in \{\pm \}} |  g^{\iota}(t) - (g_*^{\iota}(t) + g_\infty^{\iota})| \le C (\delta + \sqrt{ q(2L)}) e^{-\lambda t}+ C \sqrt{ q(2L)} \int_t^{+\infty} \kappa(s) ds.
\end{equation}
\end{theorem}

In view of the definition \eqref{eq_y_phi}, both assumptions \eqref{eq:ass_Y} and \eqref{def:kappa} are about the applied field $h$. As it will be seen from Lemma \ref{lem:kappaL1}, $\kappa(t) \to 0$ as $t \to +\infty$ and is integrable in time, so that
the estimates \eqref{est:conv_m}-\eqref{est:conv_g} yield convergence results. Notice that $\kappa$ is strongly related to $h$, whereas $\lambda$ is essentially a coercivity constant, which depends on $\gamma$ (it is related to the closeness of $|\gamma|$ to $1$). The decay functions $e^{-\lambda t}$ and $\kappa(t)$ are therefore mostly unrelated, even though in most cases (for example, as soon as $h \to 0$), $\kappa(t) \gg  e^{-\lambda t}$. 
On the other hand, if $h \le h_0 <0$ is bounded away from $0$ (for some constant $h_0 <0$), then $\kappa$ tends to $0$ with a exponential rate, and so are the convergences in \eqref{est:conv_m}-\eqref{est:conv_g}.
 
Theorem \ref{th1} therefore quantifies how and under which condition the structure made of two decoupled domain walls persists over time. Assumption \eqref{eq:ass_h_small} ensures that the external magnetic field is not too strong: this is required even for configuration with one domain wall not to be destroyed or for the constant magnetizations $\pm e_1$ to be stable. Our second assumption \eqref{eq:ass_Y} states that the free evolution of the center of the domain wall should separate them indefinitely: in order to have asymptotic stability (that is convergence of the gauge $g^\pm$), a requirement of the type $y_* \to +\infty$ is in order. It turns out that, for our analysis to work, we need a somewhat stronger integrability condition, which however remains rather mild (see point 2) of Lemma \ref{lem:y->0}).

This result is a stability statement for well prepared data, which bear some resemblance with the stability of the sum of decoupled solitons for non linear dispersive model: we refer for example to \cite{MMT02} for the generalised Korteweg-de Vries equation, or to \cite{MMT06} for the nonlinear Schrödinger equation. An important difference though, is that in these settings, each soliton brings its own dynamic, which is leading order (solitons are assumed to have distinct speeds), whereas in the present context, the dynamics is determined at the main order by the external magnetic field represented by $h$.
 
\bigskip
 
Our analysis relies on the framework developed in \cite{Cote_Ignat__stab_DW_LLG_DM}, which combines modulation techniques to split the evolution between some geometric parameters (the gauge) and a remainder term; energy estimates to control the remainder; and dynamical arguments (consequence of energy dissipation) for the gauge. 

An important point of this paper, and a novelty with respect to \cite{Cote_Ignat__stab_DW_LLG_DM}, is that this framework is amenable to space localization, and is therefore suitable to study the interactions of domain walls: we believe that is much less so for earlier methods and results on stability of (single) domain wall (referred to at the beginning of this section), which relied on spectral properties of the linearized \eqref{eq:llg} flow around domain walls. We localize the coercivity properties of the energy around each domain walls, as well as the energy dissipation equality. For these two results to make sense, one must first modulate around a sum of two domain walls; this must be done carefully, keeping into account the geometric constraint $|m|=1$ on the magnetization: this property is crucial for the coercivity. These three results are stated at the beginning of section \ref{sec:stab}, and proven in sections \ref{sec:en_coer}, \ref{sec:en_diss} and \ref{sec:decomp} respectively. Section \ref{sec:local} gives some preliminary results, in particular about a frame adapted to the domain wall and the control of the nonlinearity, and first introduced in \cite{Cote_Ignat__stab_DW_LLG_DM}. 

The proof of stability is done in section \ref{sec:stab}, and consists in a bootstrap argument: on a time interval on which one can modulate the magnetization around two domain walls, and one has sufficient control on the gauges involved and the remainder terms, we combine the localized energy dissipation and coercivity to improve these controls. We give a special attention to the decay of the remainder term in order to make the assumption on the external field $h$ as mild as possible: this is a delicate part of the analysis. We gather useful notations in an index at the end of the paper.

Before going to the proof of the main results, we conclude this section by giving some consequences on the behavior of $y_*$ of the assumptions \eqref{eq:ass_Y}-\eqref{def:kappa} for $h$.

\subsection{On the assumptions on the external field \texorpdfstring{$h$}{h}}

We recall that the distance between the two domains walls is essentially $2y_*(t)$, and it will turn out that $q(2y_*(t))$ measures correctly the interaction between them.

The assumptions on $h$ are relatively mild: apart from uniform smallness (required to ensure stability, even for one domain wall), some oscillation and decay are allowed as long as the external field still pushes the domain walls away, so that their interaction enjoys some integrability in time. This is quantified in the simple computation below. 

\begin{lem} \label{lem:y->0}
1) Assume that $h$ satisfies \eqref{eq:ass_h_small} and \eqref{eq:ass_Y} with $\delta_0 \leq \frac{\Gamma}{\alpha}$. If $t,\tau \ge 0$ are such that $|t-\tau| \le 1$, then $ |y_*(\tau) - y_*(t)| \le 1$.
As a consequence, $y_*(t) \to +\infty$ as $t \to +\infty$, and there exists $C>0$ such that for all $t,\tau \ge 0$ with $|t-\tau| \le 1$,
\begin{align} \label{est:qpm1}
q(2y_*(\tau)) \le C q(2y_*(t)).
\end{align}

2) If $\ds \liminf_{t \to +\infty} \frac{y_*(t)}{\ln t} > \frac{1}{\Gamma}$ then \eqref{eq:ass_Y} is fulfilled. This is in particular the case if $\ds \limsup_{t \rightarrow + \infty} t h(t) < - \frac{1}{\alpha}$.
\end{lem}

\begin{proof}
1) Recall that $\delta_0 \le \frac{\Gamma}{\alpha}$, so that the first bound is immediate from the mean value theorem. Assume that for some, $R \ge 1$, there exists $t_n \to +\infty$ such that $y_*(t_n) \le R$. Up to extracting in $n$, we can assume that $t_{n+1} \ge t_n +1$ for all $n$;  also, as $q$ is eventually decreasing to $0$, we can assume that $R \ge 1/\Gamma$ is so large that $\inf\{ q(2r) : r \le R+1\}$ is attained for $r = R+1$. Then, in view of the Lipschitz bound on $y_*$ induced by \eqref{eq:ass_h_small}, $y_*(t) \le R+1$ for all $t \in [t_n,t_n+1]$ so that 
\[  \int_0^\infty \sqrt{q(2y_*(t))} dt \ge \sum_{n \ge 0} \int_{t_n}^{t_{n}+1} \sqrt{ q(2y_*(t))} dt \ge \sum_{n \ge 0} \sqrt{q(2(R+1))} =+\infty,  \]
a contradiction with \eqref{eq:ass_Y}. Hence $y_*(t) \to +\infty$ as $t \to +\infty$.

Now for any $r\ge 2$ and $h \in [-2,2]$,
\[ |q(r+h) -q(r)|= |(1+r+h) e^{- \Gamma (r+h)} - (1+r) e^{- \Gamma r}| \le r e^{- \Gamma r} \abs{e^{- \Gamma h} - 1} +  \abs{h}  e^{- \Gamma (r-1)} \le C |h| q(r). \]
In particular, for any $r \ge 1$,
\[ \sup_{h \in [-1,1]} q(2(r+h)) \le C q(2r). \]
Together with the fact that for $|t-\tau| \le 1$, $|y_*(t) - y_*(\tau)| \le 1$, yields \eqref{est:qpm1}.

2) The assumption on $y_*$ writes that for some $a > 1$, and some $T  \ge 2$,
\[ \forall t \ge T, \quad y_*(t) \ge \frac{a}{\Gamma} \ln t. \]
As $q$ is eventually decreasing to $0$,
 \[ \int_T^{+\infty} \sqrt{q(2y_*(t))} \diff t \lesssim  \int_T^{+\infty} \sqrt{1+\ln t} \frac{\diff t}{t^a} <+\infty. \]
 The condition on $h$ implies that on $y_*$ by direct integration.
 \end{proof}

\begin{lem} \label{lem:kappaL1}
Under the assumptions \eqref{eq:ass_h_small} and \eqref{eq:ass_Y}, the function $\kappa$ defined in \eqref{def:kappa} has the properties:
\[ \kappa (t) \to 0 \quad \text{as} \quad  t \to +\infty \quad \text{and} \quad \ds \int_0^{+\infty} \kappa(\tau)d\tau <+\infty. \]
\end{lem}

\begin{proof}
We already saw that $y_* \to +\infty$ so that $e^{-\Gamma y_*(t)} \to 0$ and is integrable on $[0,+\infty)$. For the integral term, this is merely a convolution: as we made the hypothesis that $\sqrt{q(2y_*)}$ is integrable, convergence to zero is straightforward:
\begin{align*}
 \int_0^t e^{-2\lambda(t-s)} q(2y_*(s)) ds&  \le \| q(2y^*) \|_{L^\infty([0,+\infty))} \int_0^{t/2} e^{-2\lambda(t-s)}  + \int_{t/2}^t  q(2y_*(s)) ds \\
 &  \le e^{-\lambda t} \| q(2y^*) \|_{L^\infty([0,+\infty))} + \| \sqrt{ q(2y_*) } \|_{L^\infty([t/2,+\infty))} \int_{t/2}^{+\infty} \sqrt{q(2y_*(s))} ds \to 0.
\end{align*}
For the integrability, we need an extra ingredient: the previous Lemma \ref{lem:y->0} allows to relate to a series. To avoid side effect, first observe that there is no integrability issue on $[0,1]$:
\[ \int_0^1 \sqrt{\int_0^t e^{-2\lambda(t-s)} q(2y_*(s)) ds} dt \le \| \sqrt{q(2y_*)} \|_{L^\infty([0,1])}. \]
 If $t \in [n,n+1)$ for some integer $n \ge 1$, using \eqref{est:qpm1} there hold
\begin{align*}
\int_0^t e^{-2\lambda(t-s)} q(2y_*(s)) ds \le C \int_0^{n+1} e^{-2\lambda(n-s)} q(2y_*(s)) ds \le C \sum_{k=0}^{n} e^{-2\lambda(n-k)} q(2y_*(k))
\end{align*}
Hence, using that $\sqrt{a+b} \le \sqrt{a} + \sqrt{b}$, we infer
\begin{align*}
\sqrt{\int_0^t e^{-2\lambda(t-s)} q(2y_*(s)) ds} \le C \sum_{k=0}^{n} e^{-\lambda(n-k)} \sqrt{q(2y_*(k))}  \le C \int_0^{t} e^{-\lambda(t-s)} \sqrt{q(2y_*(s))} ds.
\end{align*}
(We used again \eqref{est:qpm1} on each interval $[k,k+1]$ for $k \le n-2$, and one last time on $[n-1,t]$ which is of length $t-n+1 \le 2 \le 2t$; this is where $t \ge 1$ is useful). Therefore,
\begin{align*}
\int_1^{+\infty} \sqrt{\int_0^t e^{-2\lambda(t-s)} q(2y_*(s)) ds} dt & \le C \int_{t \ge 1} \int_{0 \le s \le t} e^{-\lambda(t-s)} \sqrt{q(2y_*(s))} dsdt.
\end{align*}
For this last integral we split the integration domain
\[ \{ (s,t) : t \ge 1, 0 \le s \le t \} = \{ (s,t) : 0 \le s \le 1 \le t \} \cup  \{ (s,t) : 1 \le s \le t \}. \]
In both subdomains, we integrate first in $t$: there hold
\begin{align*}
 \iint_{0 \le s \le 1 \le t} e^{-\lambda(t-s)} \sqrt{q(2y_*(s))} dsdt = \frac{1}{\lambda} \int_0^1 e^{-\lambda(1-s)} \sqrt{q(2y_*(s))} ds <+\infty,
\end{align*}
and 
\begin{align*}
 \iint_{1 \le s \le t} e^{-\lambda(t-s)} \sqrt{q(2y_*(s))} dsdt = \frac{1}{\lambda} \int_1^{+\infty} \sqrt{q(2y_*(s))} ds < +\infty,
 \end{align*}
 by assumption.
\end{proof}

\section{Proof of the stability} \label{sec:stab}

\subsection{Preliminary results}

The first result we need is an appropriate modulation of the magnetization when it is near two well separated domain walls, via the use of two gauges $g^\pm$: this gives two degrees of freedom, which allows to require two orthogonality conditions, crucial to derive coervicity properties (in the next 2 Propositions), a bound on the remainder term $\varepsilon$, and the evolution laws (at main order) of the gauges.

\begin{lem}[Decomposition of the magnetization]
\label{lem:decomp_magn}
\index{Constants!$L_1$, $\delta_1$, $C_1$: related to modulation}  
\index{Functions!$g^\pm$, $y^\pm$, $\phi^\pm$: modulated gauge}
\index{Functions!$w^\pm$: modulated domain wall}
\index{Functions!$\e$: modulated error}
    There exist $\delta_1 > 0$, $L_1 \ge 1$ and $C_1 > 0$ such that the following holds. Let $T > 0$, $h \in L^\infty ((0, T))$ and $m \in \mathscr{C} ([0, T], \mathcal{H}^2)$ solution to \eqref{eq:llg}, assume that for all $t \in [0, T]$, and for some $L \ge L_1$,
    \begin{equation*}
        \delta \coloneqq \inf_{\zeta^+ \in G_{>L}, \ \zeta^- \in G_{< -L }} \norm{m (t) - \Bigl( \zeta^+ . w_*^+ + \zeta^- . w_*^- + e_1 \Bigr)}_{H^1} < \delta_1
    \end{equation*}
    Then there exists three functions :
    \begin{itemize}[label=\textbullet]
        \item $g^+ = (y^+, \phi^+) : [0, T] \rightarrow G_{>L-1} $ Lipschitz,
        \item $g^- = (y^-, \phi^-) : [0, T] \rightarrow G_{< -L+1} $ Lipschitz,
        \item $\varepsilon : [0, T] \rightarrow H^2$ continuous,
    \end{itemize}
    such that, for $w^+ = g^+ . w_*^+$ and $w^- = g^- . w_*^-$,
    \begin{itemize}
        \item $m = w^+ + w^- + e_1 + \varepsilon$,
        \item $\varepsilon$ satisfies for $\iota \in \{ \pm 1 \}$
        \begin{equation} \label{eq:orth_est1}
            \int \varepsilon \cdot \partial_x w^\iota \diff x = \int \varepsilon \cdot (e_1 \wedge w^\iota) \diff x = 0,
        \end{equation}
        %
        %
        \item the following bounds hold for all $t \in [0,T]$ and $\iota \in \{ \pm \}$:
           \begin{align} \label{eq:est_dot_g}
           \abs{\dot{g}^{\iota} (t) - \dot{g}_*^{\iota} (t)} & \leq C_1 \Bigl( \norm{\varepsilon (t)}_{H^1} + q(y^+ - y^-) \Bigr), \\
            \norm{\varepsilon (t)}_{H^1} & \leq C_1 \Bigl( \delta + q(y^+ - y^-) \Bigr). \label{est:e}
        \end{align}
        %
    %
    \end{itemize}
    %
\end{lem}

This decomposition also holds with $T = + \infty$, \emph{mutatis mutandis}. The two domain walls are at distance $y^+-y^- \ge 2L$, and \eqref{est:e} bounds the size of the perturbation. Equation \eqref{eq:est_dot_g} shows that the evolution of geometric parameters is driven by $g_*$, that is ultimately, $h$. 
This result will be proved in Section \ref{sec:decomp}.

The proof of the stability then relies on two main estimates. The first one shows an equivalence between the energy $E$ and the norm of $\varepsilon$ defined in the above modulation Lemma \ref{lem:decomp_magn}. For this, observe that $E (w_*^\sigma)$ does not depend on $\sigma \in \{ \pm 1 \}^2$, we denote it $E (w_*)$ hereafter.

\begin{prop}[Coercivity of the energy] \label{prop:equiv_energy}
\index{Constants!$L_2$, $\delta_2$, $C_2$: related to coercivity of the energy}
There exists $0 < \delta_2 \le \delta_1/3$, $L_2 \ge L_1$, $C_2>0$ and $\lambda_2 >0$ such that the following holds.
Under the assumptions (and notations) of Lemma \ref{lem:decomp_magn}, assuming further $\delta \le \delta_2$, $L \ge L_2$, there hold, for any $0< R \le L/2$ and for all $t \in [0, T]$ 
\begin{multline} \label{est:coerc_en}
        C_2 \Bigl( \norm{\varepsilon}_{H^1}^2 + ( e^{2 \Gamma (R-y^+)} + e^{2 \Gamma (R+ y^-)} ) \Bigr) \geq E(m) - 2 E(w_*) \\ \geq \Bigl( 4 \lambda_2 - \frac{C_2}{R^2} \Bigr) \norm{\varepsilon}_{H^1}^2 - C_2 \Bigl( \norm{\varepsilon}_{H^1}^3 + ( e^{2 \Gamma (R-y^+)} + e^{2 \Gamma (R + y^-)} ) \Bigr).
    \end{multline}
\end{prop}

In \eqref{est:coerc_en}, the (large but constant) $R$ is related to the size of the support of (the transition of) an adequate cut-off function $\psi_R$, defined in \eqref{def:psi_R}. The condition $R \le L/2$ ensures that each cut-off (localized around each domain wall) has only little interaction with the other domain wall. The coercivity displayed in \eqref{est:coerc_en} of course relies strongly on the orthogonality conditions \eqref{eq:orth_est1}. Proposition \ref{prop:equiv_energy} is proven in Section \ref{sec:en_coer}. 

The second result is an estimate of the evolution of the energy. It shows that, up to some quantities which are negligible enough in some sense, the energy is almost decreasing; its proof is the purpose of Section \ref{sec:en_diss}, and again, \eqref{eq:orth_est1} plays an important role.

\begin{prop}[Localised energy dissipation] \label{prop:est_dt_en}
\index{Constants!$L_3$, $\delta_3$, $C_3$: related to energy dissipation}
    There exists $0 < \delta_3 \le \delta_1/3$, $L_3 \ge L_1$, $C_3>0$ and $\lambda_3$ such that the following holds. 
    Under the same assumptions and notations as Proposition \ref{prop:equiv_energy}, assuming further $\delta \le \delta_3$, $L \ge L_3$ and for any $0< R \le L/2$, there holds, for all $t \in [0, T]$
%
%
        \begin{multline} \label{eq:est_var_en}
        \frac{\diff}{\diff t} E (m) + \left( 4\alpha \lambda_3 - \frac{C_3}{R^2} \right) \norm{\varepsilon}_{H^2}^2 \leq C_3 \Bigl( (|h| + e^{\Gamma (R - y^+)} + e^{\Gamma (R + y^-)}) \| \e \|_{H^1}^2 + \| \e \|_{H^1} \| \e \|_{H^2}^2 \Bigr) \\
       +  C_3  \Bigl( e^{2 \Gamma (R - y^+)} + e^{2 \Gamma (R + y^-)} +q(y^+ - y^-) \Bigr).
       \end{multline}
\end{prop}

\begin{proof}[Proof of Theorem \ref{th1}, assuming Lemma \ref{lem:decomp_magn} and Propositions \ref{prop:equiv_energy} and \ref{prop:est_dt_en}]

We assume in the following that  $m_0 \in \mathcal{H}^2$, so that all the computations are justified. When $m_0 \in \mathcal{H}^1$, one can use a limiting argument as in \cite[Section 4.4]{Cote_Ignat__stab_DW_LLG_DM}; we will not develop it further here.

\medskip

\emph{Step 1. Main boostrap}

Let $T_+ (m)$ be the maximum time of existence of the solution $m$ to \eqref{eq:llg}.
Let $\delta_0, L_0 >0$, and $M \geq 2$ be an extra large parameter to be fixed later. Assume already that $M\delta_0 \le \min(\delta_2,\delta_3)$, and $L_0 > \max(L_1,L_2)+2$ so large that $q(2L_0) \le 1$ and $q$ is non increasing on $[L_0,+\infty)$. Define $T_1$ as the supremum of
\begin{equation} \label{def:T1} \index{Constants!$M$: main large bootstrap constant} \index{Constants!$T_1,T$: main bootstrap exit times, $T \le T_1$}
    \left\{ T \in (0, T_+ (m)) : \forall t \in [0, T], \inf_{\zeta^{+} \in G_{>L-1}, \ \zeta^{-} \in G_{<-L+1}} \norm{m (t) - ( \zeta^{+} . w_*^{+} + \zeta^{-} . w_*^{-} + e_1 )}_{H^1} < M (\delta +\sqrt{q(2L)}) \right\}.
\end{equation}

By continuity of the flow of \eqref{eq:llg} and the assumption on the initial data as $M>1$ and $L \ge L_0 > L_1$, the above set is non empty and $T_1 > 0$. We aim at proving that $T_1 = T_+ (m) = + \infty$.

As $M\delta_0 \le \delta_1$, $m$ satisfies the assumptions of Lemma \ref{lem:decomp_magn} on $[0, T]$ for any $T \in (0,T_1)$: it provides us with the functions $g^{+} = (y^{+}, \phi^{+})$, $g^{-} = (y^{-}, \phi^{-})$ and $\varepsilon$ satisfying its conclusions, on the interval $[0,T_1)$.  Also, at time $0$, we have the improved bound:
\begin{equation} \label{est:e0}
\|\e (0) \|_{H^1} \le C_1 (\delta + q(2(L-1))) \le \tilde C_1 (\delta+ q(2L)).
\end{equation}
(where $\tilde C_1$ depends on $C_1$ only). We recall that the domain walls with initial data $g^{\pm} (0).w_*^{\pm}$ have center $y_*^{\pm}(t) = \pm y_*(t) + y^{\pm} (0)$ where 
\begin{equation*}
    y_* (t) = - \frac{\alpha}{\Gamma} \int_0^t h(s) \diff s.
\end{equation*}
Let $T_2$ be the supremum of
\begin{equation} \label{def:T2}
    \left\{ T' \in (0, T_1) : \forall t \in [0, T'], \forall \iota \in \{ \pm \}, \ \abs{y^{\iota} (t) - (y^{\iota} (0) + \iota y_* (t))} < 1 \right\}.
\end{equation}
By continuity of $y_*$ and $y^{\pm}$, we know that $T_2 > 0$
(we will show that $T_2 = T_1$ as well). We choose $T \in (0,T_2]$, and we work on the interval $[0, T]$.

\medskip

\emph{Step 2. Deriving convenient bounds on $\e$ and $g^{\pm}$}.

First observe that for $t \in [0,T]$, $y^+(t) \ge L + y_*(t)-2$ and $y^-(t) \le -L - y_*(t)+2$ so that for some constant $K$ (depending solely on $\gamma$)
\begin{gather}
e^{-2\Gamma y^+} \le K e^{-2 \Gamma (L+y_*)}, \quad e^{2\Gamma y^-} \le K e^{-2 \Gamma (L+y_*)}, \\
 q(y^{+} - y^{-}) \le q(2y_* + 2L-4) \le e^{-2\Gamma (L-2)} q(2 y_*) + q(2 (L-2)) e^{-2\Gamma y_*} \le K q(2 L)  q(2 y_*). \label{est:qy}
\end{gather}
This allows to take care of the terms in $y^{\pm}$ in the estimates.

We now choose $R$ so that we gain some coercivity in \eqref{est:coerc_en} and \eqref{eq:est_var_en}. For this we impose that
\begin{gather} \label{est:d0_1}
\delta_0 \le \delta_4:= \min \left( \sqrt{\frac{\lambda_2}{4C_2}},  \sqrt{\frac{\alpha \lambda_3}{4C_3}} \right) \quad \text{and} \quad R = \frac{1}{2\delta_4},
\end{gather}
so that $\frac{C_3}{R^2} \le \alpha \lambda_3$ and $\frac{C_2}{R^2} \le \lambda_2$.

Hence for  $C_4 = \max(K,C_2,C_3) R^2 e^{2\Gamma R}$ (the point is to observe that it does not depend on $M$), for all $t \in [0,T]$ there hold
\begin{align*}
E(m) - 2 E(w_*) & \geq 2 \lambda_2 \norm{\varepsilon}_{H^1}^2 - C_4 \Bigl(  \norm{\varepsilon}_{H^1}^3 + e^{-2\Gamma L} e^{-2 \Gamma y_*} \Bigr), \\
E(m) - 2 E(w_*) & \le C_4 \Bigl( \norm{\varepsilon}_{H^1}^2 +  e^{- 2 \Gamma L} e^{- 2 \Gamma y_*} \Bigr), \\
 \frac{\diff}{\diff t} E (m) + 3\alpha \lambda_3  \norm{\varepsilon}_{H^2}^2 & \leq C_4 (|h| \| \e \|_{H^1}^2 + \| \e \|_{H^1} \| \e \|_{H^2}^2)  + q(2L)  q(2y_*) \Bigr).
\end{align*}

We also want to make use of the smallness of $h$ and $\e$ to get rid of terms which are cubic or higher in $(\e,h)$. We therefore assume that
\begin{align} \label{est:d0_2}
\delta_0 \le \frac{\alpha \lambda_3}{C_3},
\end{align}
so that for all $t \ge 0$, $C_3 |h(t)| \le \alpha \lambda_3$. Recall that $y_* \to +\infty$, so that $\inf y_* > -\infty$: we choose $L_0$ such that
\[ L_0 \ge - 2 \inf y_*+4. \]
Then on $[0,T]$, $y^+ - y^- \ge 2 (L + y_*-2) \ge  L \ge L_0$ and as $q$ is decreasing on $[L_0,+\infty)$, $q(y^{+}-y^{-}) \le q(L)$. Thus, due to \eqref{est:e}
\[ \| \e \|_{H^1} \le C_1 M (\delta + \sqrt{q(2L)}) + C_1 q(L) \le C_1 (M+1) (\delta + q(L)). \]
We therefore assume that $\delta_0 \le \delta_5$ and $L_0 \ge L_5$ where $\delta_5 >0$ and $L_5 >0$ are such that
\begin{align} \label{est:d0_3}
\delta_5 + q (L_5) \le \frac{\min(\lambda_2,\alpha \lambda_3)}{(M+1) C_1 C_4}
\end{align}
and we infer that on $[0,T]$
\[ C_4 \| \e \|_{H^1} \le \min( \lambda_2, \alpha \lambda_3). \]
Therefore, we obtained that
\begin{align}
E(m) - 2 E(w_*) & \geq \lambda_2 \norm{\varepsilon}_{H^1}^2 - C_4 e^{-2 \Gamma L}  e^{-2\Gamma y_*} \label{est:coerc_en2} \\
E(m) - 2 E(w_*) & \le C_4 \Bigl( \norm{\varepsilon}_{H^1}^2 +  e^{-2 \Gamma L}  e^{-2\Gamma y_*}   \Bigr)\label{est:coerc_en3}  \\
 \frac{\diff}{\diff t} E (m) + \alpha \lambda_3  \norm{\varepsilon}_{H^2}^2 & \leq C_4  q(2L)  q(2y_*). \label{eq:est_var_en2}
\end{align}
The point is that, even if they hold in a regime where the relevant quantities are small or large depending on $M$, these estimates (and the constants involved) \emph{do not depend on} $M$.

\medskip

\pagebreak

\emph{Step 3. Decay of $\e$}.

Let $\tau,t \in [0,T]$ such that $\tau \le t$. Integrating \eqref{eq:est_var_en2} on $[\tau,t]$, we infer
\[ E(m(t)) + \alpha \lambda_3 \int_\tau^t \norm{\varepsilon}_{H^2}^2 \le E(m(\tau)) + C_4  q(2L) \int_\tau^t q(2y_*(s)) ds. \]
From there, together with \eqref{est:coerc_en2} and \eqref{est:coerc_en3}, we infer
\begin{align}
\lambda_2 \norm{\varepsilon(t)} _{H^1}^2 +  \alpha \lambda_3 \int_\tau^t \norm{\varepsilon(s)}_{H^2}^2 ds & \le C_4 \left( \| \e (\tau) \|_{H^1}^2 +  e^{-2 \Gamma L} (e^{-2\Gamma y_*(\tau)}+ e^{-2\Gamma y_*(t)}) + q(2L) \int_\tau^t q(2y_*(s)) ds \right) \nonumber \\
& \le 2C_4 ( \| \e (\tau) \|_{H^1}^2 +  q(2L)  \kappa_0(\tau,t)), \label{est:e_decay}
\end{align}
where for $\tau \le t$,
\[ \kappa_0(\tau,t) := e^{- 2 \Gamma y_*(\tau)}+ e^{- 2 \Gamma y_*(t)} + \int_\tau^t q(2y_*(s)) ds. \]
In particular, with $\tau=0$, we obtain a uniform bound:
\begin{gather} \label{est:e_bound_2}
\lambda_2 \norm{\varepsilon(t)} _{H^1}^2 +  \alpha \lambda_3 \int_0^t \norm{\varepsilon(s)}_{H^2}^2 ds \le 2C_4 \left( \| \e(0) \|_{H^1}^2 + q(2L)  \kappa_0(0,+\infty) \right), \\
\text{where} \quad  \kappa_0(0,+\infty) = 1+\int_0^{+\infty}q(2y_*(s)) ds <+\infty. 
\end{gather}
Going back to \eqref{est:e_decay}, fixing for now $t$ and seeing $\tau$ as a variable, we have
\begin{align*}
\MoveEqLeft \frac{\partial}{\partial \tau} \left( e^{\frac{\alpha \lambda_3}{C_4} \tau} \int_\tau^t \| \e(s) \|_{H^1}^2 ds \right) = e^{\frac{\alpha \lambda_3}{C_4} \tau} \left( \frac{\alpha \lambda_3}{C_4} \int_\tau^t \| \e(s) \|_{H^1}^2 ds  - \| \e(\tau) \|_{H^1}^2 \right) \\
& \le  2C_4 e^{\frac{\alpha \lambda_3}{C_4} \tau}   q(2L)  \kappa_0(\tau,t). 
\end{align*}
Now integrate this estimate on $[0,\tau]$ (for $\tau \le t$) to get
\[ e^{\frac{\alpha \lambda_3}{C_4} \tau} \int_\tau^t \| \e(s) \|_{H^1}^2 ds \le \int_0^\tau \| \e(s) \|_{H^1}^2 ds + 2 C_4  q(2L) \int_0^\tau e^{\frac{\alpha \lambda_3}{C_4} s}   \kappa_0(s,t) ds. \]

Assume for now that $t \ge 1$. In view of \eqref{est:e_bound_2} with $\tau=t-1$, we infer that for  $\lambda = \frac{\alpha \lambda_3}{2C_4}$,
\begin{align*}
\MoveEqLeft \int_{t-1}^t \| \e(s) \|_{H^1}^2 ds \le e^{-2\lambda (t-1)} \frac{1}{\lambda} (\| \e(0) \|_{H^1}^2 + q(2L)  \kappa_0(0,+\infty)) + 2C_4 q(2L) \int_0^t e^{- 2\lambda (t-1-s)} \kappa_0(s,t) ds.
\end{align*}
Let $\tau \in [t-1,t]$ such that
\[ \| \e(\tau) \|_{H^1}^2 \le \int_{t-1}^t \| \e(s) \|_{H^1}^2 ds \]
(we use the mean value theorem). Then  \eqref{est:e_decay} now writes,  noticing that $\kappa_0(\tau,t) \le (e \kappa(t))^2$ (due to Lemma \ref{lem:y->0}; $\kappa$ is defined in \eqref{def:kappa}),
\begin{equation} \label{est:e_decay3}
\lambda_2 \| \e (t) \|_{H^1}^2 \le e^{-2\lambda t} \tilde C_4 (\| \e(0) \|_{H^1}^2 + q(2L)  \kappa_0(0,+\infty)) +  \tilde C_4 q(2L)  \left( \kappa(t)^2 + \int_0^t e^{-2 \lambda(t-s)} \kappa_0(s,t) ds \right),
\end{equation}
where $\tilde C_4$ depends on $C_4$ and $\lambda$ only.
Observe that
\begin{align*}
\int_0^t  e^{-2 \lambda(t-s)} \kappa_0(s,t) ds & = \int_0^t e^{-2 \lambda(t-s)} \left( e^{-2\Gamma y_*(s)} +  e^{-2\Gamma y_*(t)} + \int_s^t q(2y_*(u) du \right) ds \\
& \le \frac{1}{2\lambda} e^{-2\Gamma y_*(t)} + \int_0^t e^{-2 \lambda(t-s)} q(2y_*(s))ds +\iint_{0 \le s \le u \le t} e^{-2 \lambda(t-s)} q(2y_*(u)) du ds
\end{align*}
After integrating in $s$, notice that the last double integral is bounded by
\[ \frac{1}{2\lambda} \int_{u=0}^t  e^{-2 \lambda(t-u)}  q(2y_*(u)) du, \]
 so that
 \[ \int_0^t  e^{-2 \lambda(t-s)} \kappa_0(s,t) ds \le \left( 1 + \frac{1}{2\lambda} \right) \kappa(t)^2. \]
We can therefore simplify \eqref{est:e_decay3}: recall the initial estimate  \eqref{est:e0} on $\e(0)$, we also use that $\sqrt{a+b} \le \sqrt{a} + \sqrt{b}$ for $a,b \ge 0$, and $C_1 \ge 1$, $q(2L) \le 1$. For $\ds C_5 = \sqrt{\frac{\tilde  C_4}{\lambda_2} \max \left(\tilde C_1 (1+ \sqrt{ \kappa_0(0,+\infty)}), \sqrt{2+\frac{1}{2\lambda}}  \right)}$, we obtain the bound
\begin{align} \label{est:e_decay2}
\forall t \in [0,T], \quad \| \e(t) \|_{H^1} \le C_5 e^{-\lambda t} \delta + C_5 \sqrt{q(2L)} ( e^{-\lambda t} + \kappa(t)),
\end{align}
Let us also recall at this point that due to Lemma \ref{lem:kappaL1}, $\kappa \to 0$ at $\infty$ and is integrable on $[0,+\infty)$. In particular, $\kappa$ is bounded.

\medskip

\emph{Step 4. $T_1=T_2= T^+(m) = +\infty.$}

Let us first define $M$ and choose $\delta_0$ and $L_0$. For this, we recall \eqref{eq:est_dot_g} and \eqref{est:qy}: there holds for $\iota \in \{ \pm \}$ and $t \in [0,T_1]$
\begin{align} \label{est:y-y*}
 |y^{\iota}(t) - (y^{\iota}(0)  + \iota y_*(t)) | & \le C_1 \int_0^t \| \e(s) \|_{H^1} ds + 2C_1 q(2L) \int_0^t q(2y_*(s)) ds.
\end{align}
Define 
\begin{equation} \label{def:M1}
M_1 = 2 C_1 C_5 \left( \frac{1}{\lambda} + \int_0^{+\infty} q(2y_*(t)) dt + \int_0^{+\infty} \kappa(t) dt \right),
\end{equation}
and
\[ M = 2C_5 (1+\| \kappa \|_{L^\infty([0,+\infty))}). \]
Now that $M$ has been determined, choose $\delta_0$ and $L_0$ so as to satisfies the constraints in the previous steps, namely \eqref{est:d0_1}, \eqref{est:d0_2} and \eqref{est:d0_3}, and also so that
\[ \delta_0 \le 1/(4 M_1) \quad \text{and} \quad \sqrt{2q(L_0) } \le 1/(4 M_1). \]
From \eqref{est:y-y*} and the definition of $M_1$ \eqref{def:M1}, we obtain for $\iota \in \{ \pm \}$ and 
$t \in [0,T]$
\begin{align*}
 |y^{\iota}(t) - (y^{\iota}(0) + \iota y_*(t)) | \le C_1 M_1 (\delta+ \sqrt{2q(L)}) \le \frac{1}{2} <1.
\end{align*}
In view of the definition of $T_2$  \eqref{def:T2}, and using a continuity argument, this last bound implies that $T_2=T_1$. Now, from \eqref{est:e_decay2} and the definition of $M$, we get that
\[ \forall t \in [0,T], \quad \| \e(t) \|_{H^1} \le \frac{M}{2} (\delta+ \sqrt{q(2L)}) .\]
Again, recalling the definition of $T_1$ \eqref{def:T1}, a continuity argument yields that $T_1 = T_+(m)$. This implies in turn that $\| m \|_{\q H^1}$ is uniformly bounded on $[0,T_+(m))$ and from the blow up criterion in Theorem \ref{th:lwp}, $T_+(m)=+\infty$.

\medskip

\emph{Step 5. Convergence of $m$ and $g$.}

The bound \eqref{est:e_decay2} now holds for all $t \ge0$, and gives a rate of convergence of $\e \to 0$, which is precisely \eqref{est:conv_m}.

Finally,  \eqref{eq:est_dot_g}  writes
\[ \forall t, \ge 0, \quad | \dot g^{\iota} -  \dot g^{\iota}_*| \le C_1( \| \e (t) \|_{H^1} + q(2L_0) q(2y_*(t))) \]
for $\iota  \in \{ \pm \}$. As $ \| \e \|_{H^1}$ and $ q(2y_*)$ are integrable in time, $g^{\iota} (t) - g^{\imath }_*(t)$ therefore admits a limit $g^{\iota}_\infty \in G$ as $t \to +\infty$, and
\begin{align*}
 \MoveEqLeft | g^{\iota} (t) - (g^{\iota}_\infty + g^{\iota}_*(t)) |  \le C_1 \int_t^{+\infty} (\| \e (s) \|_{H^1} + q(2L) q(2y_*(s))) ds \\
& \le \frac{C_1 C_5}{\lambda} (\delta+ \sqrt{q(2L)}) e^{-\lambda t} \\
& \qquad + C_1C_5 (1+  \sqrt{q(2L)} \| \sqrt{q(2y_*)} \|_{L^\infty([0,+\infty))})  \sqrt{q(2L)} \int_t^{+\infty} \kappa(s)ds. \end{align*}
and (as $q$ is bounded) this gives \eqref{est:conv_g}.
\end{proof}

\section{Localisation and associated basis} \label{sec:local}

In order to prove Propositions \ref{prop:equiv_energy} and \ref{prop:est_dt_en}, we need to introduce a localisation function $\psi_R$ and the basis related to $w_*^+$ and $w_*^-$. This is the content of this section, along with several miscellaneous notations and results. From now on, $C$ will be a universal positive constant which may change from line to line.

\subsection{Interaction of domain walls}

From the explicit formula of the domain walls, there holds the following.

\begin{lem} \label{lem:est_w_infty}
    There exists $C > 0$ which does not depend on $\gamma \in (-1 , 1)$ such that, for all $j \in \{0, 1, 2\}$ 
    \begin{align*}
        \abs{\partial_x^j (w_*^{(1,\sigma_2)} - e_1) (x)} & \leq C \, e^{- \Gamma x} & \text{if } x & \geq 0, \\
        \abs{\partial_x^j (w_*^{(1,\sigma_2)} + e_1) (x)} & \leq C \, e^{- \Gamma |x|} & \text{if } x &\leq 0.
    \end{align*}
\end{lem}
Similar estimates for $w_*^\sigma$ follow, \emph{mutatis mutandis}.

\begin{cor} \label{cor:est_w_orthogonal}
    There exists $C > 0$ such that for all $g^+ \in G_{> 0}$ and $g^- \in G_{< 0}$, for all $x \in \mathbb{R}$ and $j \in \{ 0, 1, 2 \}$, there holds
    \begin{equation*}
        \abs{\partial_x^j (g^+.w_*^{(1,\sigma_2)} - e_1)} \abs{\partial_x^j (g^- . w_*^{(-1, \sigma_2')} + e_1)} \leq C e^{- \Gamma (2 x - y^+ - y^-)} \qquad \text{if } x \geq y^+,
    \end{equation*}
    \begin{equation*}
        \abs{\partial_x^j (g^+.w_*^{(1,\sigma_2)} + e_1)} \abs{\partial_x^j (g^- . w_*^{(-1, \sigma_2')} + e_1)} \leq C e^{- \Gamma (y^+ - y^-)} \qquad \text{if } y^- \leq x \leq y^+,
    \end{equation*}
    \begin{equation*}
        \abs{\partial_x^j (g^+.w_*^{(1,\sigma_2)} + e_1)} \abs{\partial_x^j (g^- . w_*^{(-1, \sigma_2')} - e_1)} \leq e^{\Gamma (2 x - y^+ + y^-)} \qquad \text{if } x \leq y^-,
    \end{equation*}
    \begin{multline*}
        \Bigl( \abs{\partial_x (g^+.w_*^{(1,\sigma_2)})} + \abs{e_1 \wedge g^+.w_*^{(1,\sigma_2)}} \Bigr) \Bigl( \abs{\partial_x (g^- . w_*^{(-1, \sigma_2')})} + \abs{e_1 \wedge (g^- . w_*^{(-1, \sigma_2')})} \Bigr) \\ \leq C e^{- \Gamma (y^+ - y^-) } \times
        \begin{cases}
            e^{-2\Gamma (x - y^+)} &\qquad \text{if } x \geq y^+ \\
            1 &\qquad \text{if } x \in [y^-, y^+] \\
            e^{2 \Gamma (x - y^-)} &\qquad \text{if } x \leq y^-
        \end{cases}.
    \end{multline*}
    \begin{equation*}
        \abs{\partial_x (g^+.w_*^{(1,\sigma_2)})} \cdot \abs{g^- . w_*^{(-1, \sigma_2')} + e_1} \leq C e^{- \Gamma (y^+ - y^-) } \times
        \begin{cases}
            e^{-2\Gamma (x - y^+)} &\qquad \text{if } x \geq y^+ \\
            1 &\qquad \text{if } x \in [y^-, y^+] \\
            e^{\Gamma (x - y^-)} &\qquad \text{if } x \leq y^-
        \end{cases}.
    \end{equation*}
    \begin{equation*}
        \abs{g^+.w_*^{(1,\sigma_2)} + e_1} \abs{g^- . (\beta_* w_*^{(-1, \sigma_2')})} \leq C e^{- \Gamma (y^+ - y^-) } \times
        \begin{cases}
            e^{- 2 \Gamma (x - y^+)} &\qquad \text{if } x \geq y^+ \\
            e^{- \Gamma (x - y^-)} &\qquad \text{if } x \in [- y^-, y^+] \\
            e^{2 \Gamma (x - y^-)} &\qquad \text{if } x \leq y^-
        \end{cases}.
    \end{equation*}
\end{cor}

\begin{cor} \label{cor:est_w_orthogonal2}
The integrals over $x \in \m R$ of all the functions appearing in Corollary \ref{cor:est_w_orthogonal} are bounded (up to a constant) by $q(y^+-y^-)$.
\end{cor}

\begin{lem}[{\cite[Lemma~3.2]{Cote_Ignat__stab_DW_LLG_DM}}] \label{lem:int_w}
    There holds
    \begin{align*}
        \norm{\partial_x w_*^\sigma}_{L^2}^2 = \norm{e_1 \wedge w_*^\sigma}_{L^2}^2 &= \frac{2}{\Gamma}, \\
        \int (e_1 \wedge w_*^\sigma) \cdot \partial_x w_*^\sigma \diff x &= - \frac{2 \gamma}{\Gamma}, \\
        \int (w_*^\sigma \wedge (w_*^\sigma \wedge e_1)) \cdot \partial_x w_*^\sigma \diff x &= - 2 \sigma_1.
    \end{align*}
\end{lem}

\begin{lem} \label{lem:est_g_w}
    For all $g \in G$, there holds
    \begin{equation*}
        \norm{g.w_*^\sigma - w_*^\sigma}_{H^1} \leq C \abs{g},
    \end{equation*}
    and, for any $x \in \mathbb{R}$, if $y \ge 0$ and $j=0, 1$,
    \begin{equation*}
        \abs{\partial_x^j g.w_*^\sigma (x) - \partial_x^j w_*^\sigma (x)} \leq C \max(\abs{g}, 1) \times \begin{cases}
            e^{- \Gamma (x - y)} &\qquad \text{if } x \geq y \\
            1 &\qquad \text{if } x \in [0, y] \\
            e^{\Gamma x} &\qquad \text{if } x \leq 0
        \end{cases},
    \end{equation*}
    and similarly when $y \leq 0$.
    Moreover there exists $\tilde \delta_0 > 0$ independent of $g$ such that, if $\norm{g.w_*^\sigma - w_*^\sigma}_{H^1} \leq \tilde \delta_0$, then
    \begin{equation*}
        \abs{g} \le C \norm{g.w_*^\sigma - w_*^\sigma}_{H^1}.
    \end{equation*}
\end{lem}

\begin{proof}
    The third point is \cite[Claim~4.12]{Cote_Ignat__stab_DW_LLG_DM}. The first point was also proved in \cite{Cote_Ignat__stab_DW_LLG_DM}, see (4.6) in there. For the second point, we can refine the latter:
    \begin{align*}
        g.w_*^\sigma (x) - w_*^\sigma (x) &= \tau_y R_\phi w_*^\sigma (x) - R_\phi w_*^\sigma (x) + R_\phi w_*^\sigma (x) - w_*^\sigma (x), \\
        \abs{g.w_*^\sigma (x) - w_*^\sigma (x)} &\leq \abs{\tau_y R_\phi w_*^\sigma (x) - R_\phi w_*^\sigma (x)} + \abs{R_\phi w_*^\sigma (x) - w_*^\sigma (x)} \\
            &\leq \abs{y} \max_{[x-y, x]} \abs{\partial_x w_*^\sigma} + \abs{\phi}_{\sfrac{\mathbb{R}}{2 \pi \mathbb{Z}}} \abs{e_1 \wedge w_*^\sigma (x)},
    \end{align*}
    and the conclusion comes from Lemma \ref{lem:est_w_infty}. On the other hand, one can also estimate in another way for $x \geq 0$:
    \begin{equation*}
        g.w_*^\sigma (x) - w_*^\sigma (x) = g.w_*^\sigma (x) - \sigma_1 e_1 - (w_*^\sigma (x) - \sigma_1 e_1),
    \end{equation*}
    and we can estimate thanks to Lemma \ref{lem:est_w_infty}. A similar computation can be done for $x \leq - y$. Finally, we also have
    \begin{equation*}
        \norm{g.w_*^\sigma (x) - w_*^\sigma (x)}_{L^\infty} \leq \norm{g.w_*^\sigma}_{L^\infty} + \norm{w_*^\sigma}_{L^\infty} \leq 2,
    \end{equation*}
    which gives the estimate for $x \in [-y, 0]$. Similar arguments for the derivative give the conclusion.
\end{proof}

\subsection{Localisation}

We fix some function $\psi$ which satisfies the following assumptions :
\begin{itemize}
    \item $\psi \equiv 0$ on $(- \infty, -1]$,  $\psi \equiv 1$ on $[1, + \infty)$,
    \item $0 \leq \psi \leq 1$ on $\mathbb{R}$,
    \item $\forall x \in \mathbb{R}, 1 - \psi (x) = \psi(-x)$
    \item $\sqrt{\psi} \in W^{3, \infty} (\mathbb{R})$.
\end{itemize}
Then, we take some $R \geq 1$ large to be fixed later and we define a localisation function  and a localised scalar product:
\begin{equation} \label{def:psi_R} 
\index{Functions!$\psi$, $\psi_R$: localization functions}
\psi_R (x) \coloneqq \psi \Bigl( \ds \frac{x}{R} \Bigr), \quad \text{and} \quad   (f, g)_{\psi_R} \coloneqq \int f(x) g(x) \psi_R (x) \diff x,
\end{equation}
defined for all $f, g \in L^2 (\psi_R (x) \diff x) \coloneqq \{ h \in L^2_\text{loc} (\supp \psi_R), \norm{h}_{L^2 (\psi_R (x) \diff x)}^2 \coloneqq (h, h)_{\psi_R} < \infty \}$. 

In the course of gaining control of localized quantities, we will also use $L^2(\supp \psi_R)$ (or variants), and we emphasize the need to pay attention to the difference between
\[ \| f \|_{L^2(\psi_R(x) \diff x)} : = \left( \int |f(x)|^2 \psi_R(x) dx \right)^{1/2} \quad \text{and} \quad \| f \|_{L^2(\supp \psi_R)} : = \left( \int_{\supp \psi_R} |f(x)|^2 \psi(x) \diff x \right)^{1/2}. \]
To avoid confusion, we precise the underlying Lebesgue measure $\diff x$ when a weight function is involved.

We can also define, in a similar way, $H^k (\psi_R (x) \diff x)$ and $H^k (\supp{\psi_R})$ for all $k \in \mathbb{N}$.
Moreover, observe that for all $R \geq 1$ and integer $k$, there hold
\begin{equation} \label{eq:est_phi_A}
    \norm{\partial^k_x \psi_R}_{L^\infty} = \frac{1}{R^k} \norm{\partial^k_x \psi}_{L^\infty} \quad \text{and} \quad
    \norm{\partial^k_x ( \sqrt{\psi_R} )}_{L^\infty} = \frac{1}{R^k} \norm{\partial^k_x \sqrt{\psi}}_{L^\infty}.
\end{equation}

We also show a result with respect to the localised $H^1$ and $H^2$ norms.

\begin{lem} \label{lem:localised_norms}
\index{Functions!$\psi_0$: translated localization function}
    There exists $C > 0$ such that, for any $f \in H^{k} (\mathbb{R})$ and any $y_0$ with $\psi_0 \coloneqq \tau_{y_0} \psi_R$, there holds for $k = 1, 2$
    \begin{equation*}
        \abs{\norm{\sqrt{\psi_0} f}_{H^k}^2 - \norm{f}_{H^k (\psi_0 (x) \diff x)}^2} \leq \frac{C}{R^2} \norm{f}_{H^{k-1} (\supp \partial_x \psi_0)}^2.
    \end{equation*}
\end{lem}

\begin{rem}
    For the case $k=0$, we even have the equality
    \begin{equation*}
        \norm{\sqrt{\psi_0} f}_{L^2}^2 = \norm{f}_{L^2 (\psi_0 (x) \diff x)}^2.
    \end{equation*}
\end{rem}

\begin{proof}
    For the homogeneous $H^1$ (semi-)norm, we compute :
    \begin{equation*}
        \partial_x \Bigl( \sqrt{\psi_0} f \Bigr) = \sqrt{\psi_0} \partial_x f + f \partial_x \sqrt{\psi_0}.
    \end{equation*}
    Therefore,
    \begin{align*}
        \norm{\partial_x \Bigl( \sqrt{\psi_0} f \Bigr)}_{L^2}^2 &= \norm{\partial_x f}_{L^2 (\psi_0 (x) \diff x)}^2 + 2 \int \sqrt{\psi_0} \partial_x f \cdot f \partial_x \sqrt{\psi_0} \diff x + \int \abs{f \partial_x \sqrt{\psi_0}}^2 \diff x \\
            &= \norm{\partial_x f}_{L^2 (\psi_0 (x) \diff x)}^2 + \frac{1}{2} \int \partial_x \abs{f}^2 \partial_x \psi_0 \diff x + \int \abs{f}^2 \Bigl(\partial_x \sqrt{\psi_0}\Bigr)^2 \diff x \\
            &= \norm{\partial_x f}_{L^2 (\psi_0 (x) \diff x)}^2 + \int \abs{f}^2 \Bigl[\Bigl(\partial_x \sqrt{\psi_0}\Bigr)^2 - \frac{1}{2} \partial^2_{xx} \psi_0 \Bigr] \diff x.
    \end{align*}
    The conclusion easily follows from the estimate of the $L^\infty$ norm of $\Bigl(\partial_x \sqrt{\psi_0}\Bigr)^2 - \frac{1}{2} \partial^2_{xx} \psi_0$ with \eqref{eq:est_phi_A}.
    Morover, there holds
    \begin{equation*}
        \partial^2_{xx} \Bigl( \sqrt{\psi_0} f \Bigr) = \sqrt{\psi_0} \partial^2_{xx} f + 2 \partial_x \sqrt{\psi_0} \partial_x f + f \partial^2_{xx} \sqrt{\psi_0},
    \end{equation*}
    so that
    \begin{equation*}
        \Bigl( \partial^2_{xx} \Bigl( \sqrt{\psi_0} f \Bigr) \Bigr)^2 - \psi_0 (\partial^2_{xx} f)^2 = 2 \sqrt{\psi_0} \partial^2_{xx} f \Bigl( 2 \partial_x \sqrt{\psi_0} \partial_x f + f \partial^2_{xx} \sqrt{\psi_0} \Bigr) + \Bigl( 2 \partial_x \sqrt{\psi_0} \partial_x f + f \partial^2_{xx} \sqrt{\psi_0} \Bigr)^2.
    \end{equation*}
    Expanding the first term of the right-hand side and integrating, we get
    \begin{align}
         \int \sqrt{\psi_0} \partial^2_{xx} f \partial^2_{xx} \sqrt{\psi_0} f \diff x & = \int f \partial^2_{xx} f \Bigl(\partial^2_{xx} \psi_0 - \Bigl(\partial_x \sqrt{\psi_0}\Bigr)^2\Bigr) \diff x \\
            &= - \int (\partial_{x} f)^2 \Bigl(\partial^2_{xx} \psi_0 - \Bigl(\partial_x \sqrt{\psi_0}\Bigr)^2\Bigr) \diff x - \int f \partial_x f \partial_x \Bigl(\partial^2_{xx} \psi_0 - \Bigl(\partial_x \sqrt{\psi_0}\Bigr)^2\Bigr) \diff x, \nonumber \\
         \int \sqrt{\psi_0} \partial^2_{xx} f \partial_x \sqrt{\psi_0} \partial_x f \diff x & = \frac{1}{4} \int \partial_x \psi_0 \partial_x (\partial_x f)^2 \diff x = - \frac{1}{4} \int \partial^2_{xx} \psi_0 (\partial_x f)^2 \diff x, 
    \end{align}
    and the conclusion follows from obvious estimation.
\end{proof}

\subsection{Localized multilinear estimates in Sobolev spaces}

\begin{defi}
\index{Spaces!$O_k^\ell(f)$: polynomial terms of degree $\ell$ with at most $k$ derivatives}
    For $k \geq 0$ and $\ell \geq 1$, and given a (possibly vector valued) function $f = (f_j)_{1 \leq j \leq J}$, we use the notation
    \begin{equation*}
        g = O_k^\ell (f)
    \end{equation*}
    for a (possibly vector valued) function $g$ if each component of $g$ is an homogeneous polynomial of degree $\ell$ in the components of $f$ and their derivatives such that the total number of derivatives in each term is at most $k$, and whose coefficients are $\mathscr{C}_b^\infty (\mathbb{R})$ functions. $g$ is then the sum of terms of the form
    \begin{equation*}
        \alpha \prod_{j=1}^J \prod_{\kappa=0}^{k} (\partial_x^\kappa f_j)^{\ell_{j, \kappa}}, \qquad
        \textnormal{where} \quad
        \sum_{j, \kappa} \ell_{j, \kappa} = \ell, \quad
        \textnormal{and} \quad
        \sum_{j, \kappa} l_{j, \kappa} \kappa \leq k, \quad
        \textnormal{and} \quad
        \alpha \in \mathscr{C}_b^\infty.
    \end{equation*}
\end{defi}

\begin{lem}
    \begin{enumerate}
        \item If $k' \geq k$, then $O_k^\ell (f) = O_{k'}^\ell (f)$.
        \item If $\alpha \in \mathscr{C}^\infty_b$, then $\alpha O_k^\ell (f) = O_k^\ell (f)$.
        \item $O_k^\ell (f_1) O_{k'}^{\ell'} (f_2) = O_{k + k'}^{\ell + \ell'} (f_1, f_2)$.
        \item $\partial_x O_k^\ell (f) = O_{k+1}^\ell (f)$,
        \item $O_k^\ell (f_1 + f_2) = O_k^\ell (f_1, f_2)$.
    \end{enumerate}
\end{lem}

This notation has been used in \cite{Cote_Ignat__stab_DW_LLG_DM} to express pointwise bounds that turn into Sobolev bounds with linear dependence in the highest term. We will generalize these estimates for localised integrations :

\begin{lem} \label{lem:tech_o_lem}
    \begin{enumerate}
        \item Assume $g = O_k^\ell (f)$. Then there holds if $k \geq 2$
        \begin{equation*}
            \norm{g}_{L^2 (\supp \psi_R)} \lesssim \norm{f}_{H^k (\supp \psi_R)} \norm{f}_{H^{k-1} (\supp \psi_R)}^{\ell - 1}.
        \end{equation*}
        If $k=1$,
        \begin{equation*}
            \norm{g}_{L^2 (\supp \psi_R)} \lesssim \norm{f}_{H^1 (\supp \psi_R)}^{\ell}.
        \end{equation*}
        \item If $f \in H^1$, we have for $\ell \geq 2$,
        \begin{equation*}
            \abs{\int O_2^\ell (f) \psi_R (x) \diff x} \lesssim \norm{f}_{H^1 (\supp \psi_R)}^\ell,
        \end{equation*}
        and, if $g \in H^1$,
        \begin{equation*}
            \abs{\int O_2^1 (f) g(x) \psi_R (x) \diff x} \lesssim \norm{f}_{H^1 (\supp \psi_R)} \norm{g}_{H^1 (\supp \psi_R)}.
        \end{equation*}
        \item If $f \in H^2$, we have for $\ell \geq 2$
        \begin{equation*}
            \abs{\int O_3^\ell (f) \psi_R (x) \diff x} \lesssim \norm{f}_{H^1 (\supp \psi_R)}^{\ell - 1} \norm{f}_{H^2 (\supp \psi_R)},
        \end{equation*}
        \begin{equation*}
            \abs{\int O_4^\ell (f) \psi_R (x) \diff x} \lesssim \norm{f}_{H^1 (\supp \psi_R)}^{\ell - 2} \norm{f}_{H^2 (\supp \psi_R)}^2.
        \end{equation*}
    \end{enumerate}
\end{lem}

    The proof of all these estimates is similar to that of \cite{Cote_Ignat__stab_DW_LLG_DM}, and we refer to it. We emphasize that all the integrals involved are indeed on the support of $\psi_R$, but also that for all $j \geq 1$, $H^j (\supp \psi_R) \hookrightarrow L^\infty (\supp \psi_R)$ with uniform constant since $\supp \psi_R = (-R, \infty)$ is an unbounded interval.

\subsection{Coercivity of a Schrödinger operator}

We also define the following operator for $\Gamma = \sqrt{1 - \gamma^2}$ which was already used in \cite{Cote_Ignat__stab_DW_LLG_DM} :
\begin{equation*} \index{Operator!$L_\Gamma$: main linearized operator}
    L_\Gamma v = - \partial^2_{xx} v + \Gamma^2 (\cos^2 \theta_* - \sin^2 \theta_*) v.
\end{equation*}
We recall the main properties of this operator.

\begin{lem}[{\cite[Lemma~4.10]{Cote_Ignat__stab_DW_LLG_DM}}] \label{lem:schr_op}
    $L_\Gamma$ is a self-adjoint operator on $L^2 (\mathbb{R})$ with dense domain $H^2 (\mathbb{R})$. It admits $0$ as a simple eigenvalue with eigenfunction $\sin \theta_*$, and its spectrum is $[\Gamma^2, + \infty)$. As a consequence, there exists $\lambda_0 > 0$ such that, for all $v \in H^1 (\mathbb{R})$, $(L_\Gamma v, v) \leq 2 \norm{v}_{H^1}^2$ and
    \begin{equation*}
        (L_\Gamma v, v) \geq 4 \lambda_0 \norm{v}_{H^1}^{2} - \frac{1}{\lambda_0} \Bigl( \int v \sin( \theta_* ) \diff x \Bigr)^2,
    \end{equation*}
    and for all $v \in H^2 (\mathbb{R})$,
    \begin{equation*}
        \norm{L_\Gamma v}_{L^2}^{2} \geq 4 \lambda_0 \norm{v}_{H^2}^{2} - \frac{1}{\lambda_0} \Bigl( \int v \sin( \theta_* ) \diff x \Bigr)^2.
    \end{equation*}
\end{lem}

However, we will not be able to apply directly this lemma on the same functions as in \cite{Cote_Ignat__stab_DW_LLG_DM}. Indeed, the localisation function needs to be taken into account, as follows.

\begin{lem} \label{lem:loc_schr_op}
    There exists $C > 0$ such that, for any $f \in H^2 (\mathbb{R})$ and any $y_0$ with $\psi_0 \coloneqq \tau_{y_0} \psi_R$, 
    \begin{multline*}
        2 \norm{f}_{H^1 (\psi_0 (x) \diff x)}^2 + \frac{C}{R^2} \norm{f}_{L^2 (\supp{\partial_x \psi_0})}^2 
        \geq \biggl(L_\Gamma f, f\biggr)_{\psi_0} \\
        \geq 4 \lambda_0 \norm{f}_{H^1 (\psi_0 (x) \diff x)}^2 - \frac{1}{\lambda_0} \Bigl( \int \sqrt{\psi_0} f \sin( \theta_* ) \diff x \Bigr)^2 - \frac{C}{R^2} \norm{f}_{L^2 (\supp{\partial_x \psi_0})}^2,
    \end{multline*}
    and
    \begin{equation*}
        \norm{L_\Gamma f}_{L^2 (\psi_0 (x) \diff x)}^2 \geq 4 \lambda_0 \norm{f}_{H^2 (\psi_0 (x) \diff x)}^2  - \frac{1}{\lambda_0} \Bigl( \int \sqrt{\psi_0} f \sin( \theta_* ) \diff x \Bigr)^2 - \frac{C}{R^2} \norm{f}_{H^1 (\supp \partial_x \psi_0)}.
    \end{equation*}
\end{lem}

\begin{proof}
    First, remark that we constructed $\psi_R$ so that $\sqrt{\psi_0} f \in H^2$ as soon as $f \in H^2$. Then, we also have
    \begin{equation*}
        (L_\Gamma f, f)_{\psi_0} = \biggl(\sqrt{\psi_0} L_\Gamma f, \sqrt{\psi_0} f \biggr).
    \end{equation*}
    From the definition of $L_\Gamma$, there holds
    \begin{equation} \label{eq:localised_L}
        L_\Gamma (\sqrt{\psi_0} f) - \sqrt{\psi_0} L_\Gamma f = - \partial^2_{xx} \Bigl(\sqrt{\psi_0} f\Bigr) + \sqrt{\psi_0} \partial^2_{xx} f
            = - \partial^2_{xx} \sqrt{\psi_0} f - 2 \partial_x \sqrt{\psi_0} \partial_x f.
    \end{equation}
    Thus,
    \begin{align*}
        (L_\Gamma f, f)_{\psi_0} - \biggl(L_\Gamma \Bigl( \sqrt{\psi_0} f \Bigr), \sqrt{\psi_0} f \biggr) &= \biggl( \partial^2_{xx} \sqrt{\psi_0} f, \sqrt{\psi_0} f \biggr) + 2 \biggl( \partial_{x} \sqrt{\psi_0}  \partial_x f, \sqrt{\psi_0} f \biggr) \\
            &= \biggl( \sqrt{\psi_0} \partial^2_{xx} \sqrt{\psi_0} f, f \biggr) + 2 \biggl( \sqrt{\psi_0} \partial_{x} \sqrt{\psi_0}  \partial_x f, f \biggr).
    \end{align*}
    Therefore, after an integration by parts in the last term,
    \begin{align*}
        (L_\Gamma f, f)_{\psi_0} - \biggl(L_\Gamma \Bigl( \sqrt{\psi_0} f \Bigr), \sqrt{\psi_0} f \biggr) &= \biggl( \sqrt{\psi_0} \partial^2_{xx} \sqrt{\psi_0} f, f \biggr) - \biggl( \partial_x \Bigl( \sqrt{\psi_0} \partial_{x} \sqrt{\psi_0} \Bigr) f, f \biggr) \\
            &= - \biggl( \Bigl( \partial_{x} \sqrt{\psi_0} \Bigr)^2 f, f \biggr).
    \end{align*}
    Therefore, we get
    \begin{equation*}
        \abs{\biggl(L_\Gamma f, f\biggr)_{\psi_0} - \biggl(L_\Gamma \Bigl(\sqrt{\psi_0} f\Bigr), \sqrt{\psi_0} f\biggr)} \leq \frac{C}{R^2} \norm{f}_{L^2 (\supp \partial_x \psi_0)}^2.
    \end{equation*}
    The conclusion follows by applying Lemma \ref{lem:schr_op} to $\biggl(L_\Gamma \Bigl(\sqrt{\psi_0} f\Bigr), \sqrt{\psi_0} f\biggr)$ and with Lemma \ref{lem:localised_norms}.
    As for the second estimate, from \eqref{eq:localised_L}, we also get
    \begin{equation*}
        \Bigl( L_\Gamma \Bigl(\sqrt{\psi_0} f \Bigr) \Bigr)^2 = \Bigl( L_\Gamma f \Bigr)^2 \psi_0 - 2 \sqrt{\psi_0} L_\Gamma f \Bigl( \partial^2_{xx} \sqrt{\psi_0} f + 2 \partial_x \sqrt{\psi_0} \partial_x f \Bigr) + \Bigl( \partial^2_{xx} \sqrt{\psi_0} f + 2 \partial_x \sqrt{\psi_0} \partial_x f \Bigr)^2.
    \end{equation*}
    For the second term, expanding $L_\Gamma f$, we obtain by integrating the following terms
\begin{gather}
    \begin{aligned}
      \MoveEqLeft  \int \sqrt{\psi_0} \partial^2_{xx} f \partial^2_{xx} \sqrt{\psi_0} f \diff x = \int f \partial^2_{xx} f \Bigl(\frac{1}{2} \partial^2_{xx} \psi_0 - \Bigl(\partial_x \sqrt{\psi_0}\Bigr)^2\Bigr) \diff x  \\
           & =  - \int (\partial_{x} f)^2 \Bigl(\frac{1}{2} \partial^2_{xx} \psi_0 - \Bigl(\partial_x \sqrt{\psi_0}\Bigr)^2\Bigr) \diff x - \int f \partial_x f \partial_x \Bigl(\frac{1}{2} \partial^2_{xx} \psi_0 - \Bigl(\partial_x \sqrt{\psi_0}\Bigr)^2\Bigr) \diff x,
            \end{aligned} \label{eq:term1} \\
        \int \sqrt{\psi_0} \partial^2_{xx} f \partial_x \sqrt{\psi_0} \partial_x f \diff x = \frac{1}{4} \int \partial_x \psi_0 \partial_x (\partial_x f)^2 \diff x = - \frac{1}{4} \int \partial^2_{xx} \psi_0 (\partial_x f)^2 \diff x, \label{eq:term2}
    \end{gather}
    and also
    \begin{equation*}
        \int \sqrt{\psi_0} \Gamma^2 (\cos^2 \theta_* - \sin^2 \theta_*) \partial^2_{xx} \sqrt{\psi_0} f^2 \diff x
        \qquad \text{and} \qquad
        \frac{1}{2} \int \Gamma^2 (\cos^2 \theta_* - \sin^2 \theta_*) \partial_x \psi_0 f \partial_x f \diff x.
    \end{equation*}
    From straightforward estimates thanks to \eqref{eq:est_phi_A}, we get
    \begin{equation*}
        \abs{\norm{L_\Gamma f}_{L^2 (\psi_0 (x) \diff x)}^2 - \norm{L_\Gamma \Bigl( \sqrt{\psi_0} f \Bigr)}_{L^2}^2} \leq \frac{C}{R} \norm{f}_{H^1 (\supp \partial_x \psi_0)}^2.
    \end{equation*}
    The estimate then comes by applying Lemma \ref{lem:schr_op} to $\norm{L_\Gamma \Bigl( \sqrt{\psi_0} f \Bigr)}_{L^2}^2$ and Lemma \ref{lem:localised_norms} again.
\end{proof}

\subsection{Expansion in the associated basis}

The computations made in \cite{Cote_Ignat__stab_DW_LLG_DM} show that the following frame is better adapted to a $\mathbb{S}^2$-valued magnetisation $m$ close to a domain wall $w_*^\sigma$ for some $\sigma = (\sigma_1, \sigma_2) \in \{ \pm 1\}^2$. Define
\begin{equation*} \index{Functions!$n_*$, $p_*$: coordinates of the adapted frame}
    n_*^\sigma \coloneqq - \frac{1}{\sin \theta_*} w_*^\sigma \wedge (e_1 \wedge w_*^\sigma), \qquad
    p_*^\sigma \coloneqq w_*^\sigma \wedge n_*.
\end{equation*}
$(w_* (x), n_* (x), p_* (x))$ is thus an orthonormal basis in $\mathbb{R}^3$ for all $x \in \mathbb{R}$. 

One important observation, which motivates the introduction of this basis, is the following. Let $m = w + \eta \in \m S^2$ with $\eta$ small: if one decomposes 
\[
\index{Functions!$\mu,\nu,\rho$: coordinates of $\eta$ in the adapted frame}
\eta = \mu w_* + \nu n_* + \rho p_*, \]
then $\mu$ is quadratic in $\eta$, whose norm is thus equivalent to that of $\nu$ and $\rho$. This is a pointwise in $x$, and is can be globalized or localized.

The precise statement is as follows.

\begin{lem} \label{lem:expand_eta}
    There exists $\delta_3 > 0$ and $C_2 > 0$ such that the following holds. Let $w_* \coloneqq w_*^\sigma$ for some $\sigma = (\sigma_1, \sigma_2) \in \{ \pm 1\}^2$ be a domain wall.
    Let $m = w_* + \eta : \mathbb{R} \rightarrow \mathbb{S}^2$ and $x_0 > 0$ be such that
    \begin{equation*}
        \norm{\eta}_{H^1 ((- x_0, \infty))} < \delta_3.
    \end{equation*}
    We decompose $\eta$ in the $(w_*, n_*, p_*)$ basis pointwise in $x$:
    \begin{equation*} 
        \eta = \mu w_* + \nu n_* + \rho p_* \quad
        \text{where} \quad
        \mu \coloneqq \eta \cdot w_*, \quad
        \nu = \eta \cdot n_*, \quad
        \rho = \eta \cdot p_*.
    \end{equation*}
    Then $\mu, \nu, \rho \in H^1 ((- x_0, \infty))$, with
    \begin{equation} \label{eq:ineg1}
        \norm{\mu}_{H^1 ((- x_0, \infty))} \leq C_2 \norm{\eta}_{H^1 ((- x_0, \infty))}^2, \qquad
        \frac{1}{C_2} \norm{\eta}_{H^1 ((- x_0, \infty))} \leq \norm{(\nu, \rho)}_{H^1 ((- x_0, \infty))} \leq C_2 \norm{\eta}_{H^1 ((- x_0, \infty))}.
    \end{equation}
    Moreover, as soon as $x_0 \geq R$, there also holds
    \begin{equation} \label{eq:equiv_eta_nu}
        \norm{\mu}_{H^1 (\psi_R \diff x)} \leq C_2 \norm{\eta}_{H^1 (\psi_R \diff x)} \norm{\eta}_{H^1 (\supp \psi_R)}, \qquad
        \frac{1}{C_2} \norm{\eta}_{H^1 (\psi_R \diff x)} \leq \norm{(\nu, \rho)}_{H^1 (\psi_R \diff x)} \leq C_2 \norm{\eta}_{H^1 (\psi_R \diff x)}.
    \end{equation}
    In particular, $\mu = \frac{1}{2} \abs{\eta}^2 = O_0^2 (\eta)$.
    If furthermore $\eta \in H^2$, then $\mu, \nu, \rho \in H^2$ and
    \begin{equation*} \label{eq:equiv_eta_nu2}
        \norm{(\nu, \rho)}_{H^2 ((- x_0, \infty))} \leq C_2 \norm{\eta}_{H^2 ((- x_0, \infty))}
    \end{equation*}
    Last, there also hold
    \begin{equation} \label{eq:rel_orth}
        \rho \sin \theta_* = \eta \cdot (e_1 \wedge w_*), \qquad
        \sigma_1 \sin \theta_* \nu = \frac{1}{\Gamma^2} \eta \cdot \partial_x w_* - \gamma \eta \cdot (e_1 \wedge w_*)
    \end{equation}
\end{lem}


\begin{proof}
    The proof is similar to the first step of the proof of \cite[Proposition~4.16]{Cote_Ignat__stab_DW_LLG_DM}.
    First, the relations between $\mu$, $\nu$, $\rho$ and $\eta$
    along with Lemma \ref{lem:tech_o_lem} give
    \begin{equation*}
        \norm{\mu}_{H^k ((- x_0, \infty))} + \norm{\nu}_{H^k ((- x_0, \infty))} + \norm{\rho}_{H^k ((- x_0, \infty))} \lesssim \norm{\eta}_{H^k ((- x_0, \infty))}.
    \end{equation*}
    On the other side, $\eta = \mu w_* + \nu n_* + \rho p_*$ and therefore
    \begin{equation*}
        \norm{\eta}_{H^k ((- x_0, \infty))} \lesssim \norm{\mu}_{H^k ((- x_0, \infty))} + \norm{\nu}_{H^k ((- x_0, \infty))} + \norm{\rho}_{H^k ((- x_0, \infty))}.
    \end{equation*}
    $\mu = \frac{1}{2} \abs{\eta}^2$ comes from the expansion of $\abs{w_* + \eta}^2 = 1$, which gives the first inequality of \eqref{eq:ineg1} with Lemma \ref{lem:tech_o_lem}. As soon as $\norm{\eta}_{H^1 ((- x_0, \infty))}$ is small enough the second inequality is then straightforward.
    In a similar way, 
    we also get $\partial_x \mu = \eta \cdot \partial_x \eta$, and the first inequality of \eqref{eq:equiv_eta_nu} is then easily proved. If $\supp \psi_R \subset (- x_0, \infty)$, then $\norm{\eta}_{H^1 (\supp \psi_R)} < \delta_3$ from the assumption and the second inequality of \eqref{eq:equiv_eta_nu} is proved similarly.
    
    Eventually, the last equality comes from the formulas (see \eqref{ode_w*} for the first one)
    \begin{equation*}
        \partial_x w_* = \Gamma^2 \sin \theta_* ( \sigma_1 n_* + \gamma p_*), \qquad
        e_1 \wedge w_* = \sin \theta_* p_*. \qedhere
    \end{equation*}
\end{proof}

With this result, the magnetization can be decomposed in a similar way when it is close to a 2-domain wall structure (with the two domain walls far away enough).

\begin{lem} \label{lem:est_eta_eps}
    There exists $\delta_2' > 0$ and $L_0 > R$ such that the following holds.
    Let $L > L_0$, $g^{\pm} = (y^{\pm}, \phi^{\pm})$ such that $g^+ \in G_{> L}$ and $g^- \in G_{< - L}$. Let $m = w^+ + w^- + e_1 +  {\color{black} \varepsilon} \in \mathcal{H}^1$ for some $w^+ = g^+ . w^{(1, \sigma_2)}_*$ and $w^- = g^- . w^{(- 1, \sigma_2')}_*$, with $\varepsilon \in H^1$.
    Define also
    \begin{equation*} \index{Functions!$\eta^\pm$: error related to one domain wall $w_*^\pm$}
        \eta^\pm \coloneqq (- g^\pm).m - w^\pm_* = (- g^\pm).(w^\mp + e_1 + \varepsilon).
    \end{equation*}
    %
    $\eta^\pm$ can be decomposed in the $(w_*^\pm, n^\pm, p^\pm)$ basis associated to $w_*^\pm$ :
    \begin{equation*} \index{Functions!$\mu^\pm, \nu^\pm, \rho^\pm$: coordinates of $\eta^\pm$ in the adapted frame}
        \eta^\pm = \mu^\pm w_*^\pm + \nu^\pm n_*^\pm + \rho^\pm p_*^\pm.
    \end{equation*}
    Finally, define $\psi_R^\pm (x) = \psi_R (\pm x - y^\pm)$.
    \index{Functions!$\psi_R^\pm$: gauge translated localization functions}
    Then, for any $k \in \{ 0, 1, 2 \}$, if $m \in \mathcal{H}^k$ and $\norm{\varepsilon}_{H^1} < \delta_2'$,
    there hold
    \begin{equation} \label{eq:ineg2}
        \norm{\eta^\pm}_{H^k (\psi_R^\pm \diff x)} \leq \norm{\eta^\pm}_{H^k (\supp \psi_R^\pm)} \leq \norm{\varepsilon}_{H^k} + C e^{\Gamma (R \pm y^\mp)},
    \end{equation}
    %
    %
    \begin{equation} \label{eq:equiv_eta_eps}
        \norm{\varepsilon}_{H^k} \leq \norm{\eta^+}_{H^k (\psi_R^+ (x) \diff x)} + \norm{\eta^-}_{H^k (\psi_R^- (x) \diff x)} + C \Bigl( e^{\Gamma (R - y^+)} + e^{\Gamma (R + y^-)} \Bigr).
    \end{equation}
    %
    %
    %
    Moreover, there also holds
    \begin{equation} \label{eq:equiv_eta_nu_rho}
        \frac{1}{C} \norm{(\nu^\pm, \rho^\pm)}_{H^k (\supp \psi_R^\pm)} \leq \norm{\eta^\pm}_{H^k (\supp \psi_R^\pm)} \leq C \norm{(\nu^\pm, \rho^\pm)}_{H^k (\supp \psi_R^\pm)},
    \end{equation}
    %
\end{lem}

\begin{rem}
    This lemma shows that, as soon as $\norm{\varepsilon}_{H^1}$ is small enough, estimating $\norm{\varepsilon}_{H^k}$ is equivalent to estimating both $\norm{(\nu^\pm, \rho^\pm)}_{H^k (\supp \psi_R^\pm)}$ and $e^{\Gamma (R \pm y^\pm)}$. This property will be intensively used in the following.
\end{rem}

\begin{proof}
    The first inequality of \eqref{eq:ineg2} comes from the fact that $0 \leq \psi_R^\pm \leq 1$. The second one can be easily deduced from the following computation :
    \begin{align*}
        \norm{\eta^+}_{H^k (\supp \psi_R^+)} &= \norm{(- g^+).(w^- + e_1 + \varepsilon)}_{H^k (\supp \psi_R^+)} = \norm{w^- + e_1 + \varepsilon}_{H^k (\supp \psi_R)} \\
            &\leq \norm{w^- + e_1}_{H^k (\supp \psi_R)} + \norm{\varepsilon}_{H^k (\supp \psi_R)} \\
            &\leq \norm{g^-.(w^-_* + e_1)}_{H^k ((-R, \infty))} + \norm{\varepsilon}_{H^k} \\
            &\leq \norm{w^-_* + e_1}_{H^k ((- R - y^-, \infty))} + \norm{\varepsilon}_{H^k},
    \end{align*}
    and the conclusion with Lemma \ref{lem:est_w_infty}. The computations for $\eta^-$ are similar.
    We also have
    \begin{equation*}
        \norm{\varepsilon}_{H^k}^2 = \norm{\varepsilon}_{H^k (\psi_R (x) \diff x)}^2 + \norm{\varepsilon}_{H^k (\psi_R (-x) \diff x)}^2,
    \end{equation*}
    and, similarly,
    \begin{align*}
        \norm{\varepsilon}_{H^k (\psi_R (x) \diff x)} &= \norm{g^+.\eta^+ - (w^- + e_1)}_{H^k (\psi_R (x) \diff x)} \\
            &\leq \norm{g^+.\eta^+}_{H^k (\psi_R (x) \diff x)} + \norm{w^- + e_1}_{H^k (\psi_R (x) \diff x)} \\
            &\leq \norm{\eta^+}_{H^k (\psi_R^+ (x) \diff x)} + \norm{w^- + e_1}_{H^k (\supp \psi_R)} \\
            &\leq \norm{\eta^+}_{H^k (\psi_R^+ (x) \diff x)} + C e^{\Gamma (R + y^-)}.
    \end{align*}
    Once again, the computation for $\norm{\varepsilon}_{H^k (\psi_R (- x) \diff x)}$ is similar and symmetric.
    Eventually, \eqref{eq:equiv_eta_nu_rho} comes from Lemma \ref{lem:expand_eta} and the fact that $\supp \psi_R^+ \subset [- R - y^+, \infty)$ and $\supp \psi_R^- \subset (- \infty, R - y^-]$.
\end{proof}

The goal is to use the previous lemma with the decomposition provided by Lemma \ref{lem:decomp_magn}. However, the localisation function will still remain in the integrals we compute. Therefore, we won't be able to get the same vanishing integrals as in \cite{Cote_Ignat__stab_DW_LLG_DM} when we apply Lemma \ref{lem:schr_op}. However, the integrals we will obtain are still small enough : the reminiscence of the localisation function gives only negligible terms, as shown in the following lemma.

\begin{lem}[Almost orthogonality] \label{lem:almst_orth}
    With the same assumptions and notations as in Lemma \ref{lem:decomp_magn}, define $\eta^\pm$, $\mu^\pm$, $\nu^\pm$, $\rho^\pm$ and $\psi_R^\pm$ as in Lemma \ref{lem:est_eta_eps}. Then there holds
    \begin{equation*} 
        \abs{\int \sqrt{\psi_R^\pm} \rho^\pm \sin \theta_* \diff x} + \abs{\int \sqrt{\psi_R^\pm} \nu^\pm \sin \theta_* \diff x} \leq C \Bigl( q(y^+ - y^-) + \norm{\varepsilon}_{L^2} e^{\Gamma (R \mp y^\pm)} \Bigr)
    \end{equation*}
\end{lem}

\begin{proof}
    From \eqref{eq:rel_orth}, we get
    \begin{align*}
        \int \sqrt{\psi_R^\pm} \rho^\pm \sin \theta_* \diff x &= \int \sqrt{\psi_R^\pm} \eta^\pm \cdot (e_1 \wedge w_*^\pm) \diff x, \\
        \sigma_1 \int \sqrt{\psi_R^\pm} \nu^\pm \sin \theta_* \diff x &= \frac{1}{\Gamma} \int \sqrt{\psi_R^\pm} \eta^\pm \cdot \partial_x w_*^\pm \diff x - \gamma \int \sqrt{\psi_R^\pm} \eta^\pm \cdot (e_1 \wedge w_*^\pm) \diff x.
    \end{align*}
    On the other hand, by the expression of $\eta^\pm$,
    \begin{equation*}
        \int \sqrt{\psi_R^\pm} \eta^\pm \cdot (e_1 \wedge w_*^\pm) \diff x = \int \sqrt{\psi_R (\pm x)} (g^\mp . w_*^\mp + e_1) \cdot (e_1 \wedge g^\pm . w_*^\pm) \diff x + \int \sqrt{\psi_R (\pm x)} \, \varepsilon \cdot (e_1 \wedge g^\pm . w_*^\pm) \diff x.
    \end{equation*}
    %
    For the first term, we can estimate by using the fact that $R < L < \min{(y^+, - y^-)}$ and with Corollary \ref{cor:est_w_orthogonal}:
    \begin{align*}
        \abs{\int \sqrt{\psi_R} (g^- . w_*^- + e_1) \cdot (e_1 \wedge g^+ . w_*^+) \diff x} &\leq \int_{-R}^\infty \abs{g^- . w_*^- + e_1} \abs{e_1 \wedge g^+ . w_*^+} \diff x \\
            &\leq C (1 + y^+ - y^-) e^{- \Gamma (y^+ - y^-)}.
    \end{align*}
    For the second term, by using the orthogonality conditions \eqref{eq:orth_est1}, we get
    \begin{align*}
        \abs{\int \sqrt{\psi_R} \varepsilon \cdot (e_1 \wedge g^+ . w_*^+) \diff x} &= \abs{\int \Bigl( \sqrt{\psi_R} - 1 \Bigr) \varepsilon \cdot (e_1 \wedge g^+ . w_*^+) \diff x} \\
            &\le \int_{- \infty}^{R} \abs{\varepsilon} \abs{e_1 \wedge g^+ . w_*^+} \diff x \le C \norm{\varepsilon}_{L^2} e^{\Gamma (R - y^+)}.
    \end{align*}
    Similar estimates hold for $\ds \int \sqrt{\psi_R^-} \eta^- \cdot (e_1 \wedge w_*^-) \diff x$ and for $\ds \int \sqrt{\psi_R^\pm} \eta^\pm \cdot \partial_x w_*^\pm \diff x$, and thus does the conclusion.
\end{proof}

\section{Localised energies} \label{sec:en_coer}

In this section, we prove Proposition \ref{prop:equiv_energy}. For this, we localise the energy thanks to the localisation $\psi_R$, which is a classical technique for to study multi-solitons for nonlinear dispersive equations.

We define the localised energies :
\begin{equation*} \index{Functionals!$E^\pm$: localized energy with transition around $0$}
    E^+ (m) \coloneqq \frac{1}{2} \int \Bigl( \abs{\partial_x m}^2 + 2 \gamma \partial_x m \cdot (e_1 \wedge m) + (1 - m_1^2) \Bigr) \psi_R (x) \diff x,
\end{equation*}
\begin{equation*}
    E^- (m) \coloneqq \frac{1}{2} \int \Bigl( \abs{\partial_x m}^2 + 2 \gamma \partial_x m \cdot (e_1 \wedge m) + (1 - m_1^2) \Bigr) \psi_R (- x) \diff x
\end{equation*}
By the properties of $\psi_R$, we know that $E^+ + E^- = E$.
Then, we define the following modified energies :
\begin{equation*} \index{Functionals!$\tilde E^\pm$: localized energy, with transition not far from a domain wall}
    \tilde{E}^+ (m) \coloneqq E^+ (\tau_{y^+} m) = \frac{1}{2} \int \Bigl( \abs{\partial_x m}^2 + 2 \gamma \partial_x m \cdot (e_1 \wedge m) + (1 - m_1^2) \Bigr) \psi_R^+ (x) \diff x
\end{equation*}
where $\psi_R^+ \coloneqq \tau_{- y^+} \psi_R = \psi_R (x + y^+)$, and
\begin{equation*}
    \tilde{E}^- (m) \coloneqq E^- (\tau_{y^-} m ) = \frac{1}{2} \int \Bigl( \abs{\partial_x m}^2 - 2 \gamma \partial_x m \cdot (e_1 \wedge m) + (1 - m_1^2) \Bigr) \psi_R^- (x) \diff x,
\end{equation*}
where $\psi_R^- \coloneqq \tau_{- y^-} \psi_R (-x) = \psi_R (- x + y^-)$.

\subsection{First estimate on the localised energies}

First, we want to expand the localised energies defined previously. For this, we define $\eta^\pm$, and then $\mu^\pm$, $\nu^\pm$ and $\rho^\pm$ like in Lemma \ref{lem:est_eta_eps}. With similar computations as in \cite{Cote_Ignat__stab_DW_LLG_DM}, we show that, up to some additional negligible terms, the expansion of the localised energies gives no term of order $1$ and nice terms of order $0$ and $2$.

\begin{prop} \label{prop:energy_exp}
    Let the assumptions of Lemma \ref{lem:est_eta_eps} be satisfied.
    Then, 
\[
        \abs{E^\pm (m) - \Bigl[ \tilde{E}^\pm (w_*^\pm) + \frac{1}{2} \Bigl( (L_\Gamma \nu^\pm, \nu^\pm)_{\psi_R^\pm} + (L_\Gamma \rho^\pm, \rho^\pm)_{\psi_R^\pm} \Bigr) \Bigr] } 
        \leq C \Bigl[ \norm{\varepsilon}_{H^1}^3 + \frac{1}{R^2} \norm{\varepsilon}_{L^2}^2 + e^{2 \Gamma (R - y^+)} + e^{2 \Gamma (R + y^-)} \Bigl],
\]
\end{prop}

\begin{proof}
    The pointwise estimate of steps 2 and 3 of the proof of \cite[Proposition~4.16]{Cote_Ignat__stab_DW_LLG_DM} still hold, both for $\eta^+$ and $\eta^-$. In particular, we have :
    \begin{align}
        \delta E (\eta^\pm) & = O_2^2 (\eta^\pm) \pm 2 (\partial_{xx}^2 \theta_* \nu^\pm + \partial_x \theta_* \partial_x \nu^\pm) w_*^\pm + (- \partial_{xx}^2 \nu^\pm + \Gamma \nu^\pm) n_*^\pm + (- \partial_{xx}^2 \rho^\pm + \Gamma \rho^\pm) p_*^\pm, \\
        \label{eq:eta_var_energy_w}
        \eta^\pm \cdot \delta E (w_*^\pm) & = \beta_* \eta^\pm \cdot w_*^\pm = - \frac{1}{2} \beta_* \abs{\eta^\pm}^2, \\ \label{eq:eta_var_energy_eta}
        \eta^\pm \cdot \delta E (\eta^\pm) & = O_2^3 (\eta^\pm) - \nu^\pm \partial_{xx}^2 \nu^\pm + \Gamma^2 (\nu^\pm)^2 - \rho^\pm \partial_{xx}^2 \rho^\pm + \Gamma^2 (\rho^\pm)^2
    \end{align}
    Moreover, even if $E^\pm$ consists only in quadratic terms of $m$, it is not invariant under translation due to the localisation term $\psi_R (\pm x)$, and one should also take care about the integrations by part, so that the relations of the step 4 of the proof of \cite[Proposition~4.11]{Cote_Ignat__stab_DW_LLG_DM} are different :
    \begin{align*}
        E^\pm (m) &= \tilde{E}^\pm ( \eta^\pm + w_*^\pm ) \\
            &\begin{multlined}
                = \tilde{E}^\pm (w_*^\pm) + \int \eta^\pm \cdot \delta E (w_*^\pm) \psi_R^\pm (x) \diff x + \frac{1}{2} \int \eta^\pm \cdot \delta E (\eta^\pm) \psi_R^\pm (x) \diff x \\
                \begin{aligned}
                    &- \int \eta^\pm \cdot \partial_x w_*^\pm \partial_x \psi_R^\pm (x) \diff x - \frac{1}{2} \int \eta^\pm \cdot \partial_x \eta^\pm \partial_x \psi_R^\pm (x) \diff x \\
                    &- \gamma \int \eta^\pm \cdot (e_1 \wedge w_*^\pm) \partial_x \psi_R^\pm (x) \diff x.
                \end{aligned}
            \end{multlined}
    \end{align*}
    Using both \eqref{eq:eta_var_energy_w} and \eqref{eq:eta_var_energy_eta} along with Lemma \ref{lem:tech_o_lem}, we get
    \begin{align*}
        \MoveEqLeft E^\pm (m) - \tilde{E}^\pm (w_*^\pm) = O (\norm{\eta^\pm}_{H^1 (\supp \psi_R^\pm)}^3) \\
                &- \frac{1}{2} \int \beta_* \abs{\eta^\pm}^2 \psi_R^\pm (x) \diff x + \frac{1}{2} \int \Bigl( (- \partial^2_{xx} \nu^\pm + \Gamma^2 \nu^\pm) \nu^\pm + (- \partial_{xx}^2 \rho^\pm + \Gamma^2 \rho^\pm) \rho^\pm \Bigr) \psi_R^\pm (x) \diff x \\
                &- \int \eta^\pm \cdot \partial_x w_*^\pm \partial_x \psi_R^\pm (x) \diff x - \gamma \int \eta^\pm \cdot (e_1 \wedge w_*^\pm) \partial_x \psi_R^\pm (x) \diff x + \frac{1}{4} \int \abs{\eta^\pm}^2 \partial^2_{xx} \psi_R^\pm (x) \diff x.
    \end{align*}
    We now use the fact that
    \begin{equation*}
        \abs{\eta^\pm}^2 = (\mu^\pm)^2 + (\nu^\pm)^2 + (\rho^\pm)^2 = (\nu^\pm)^2 + (\rho^\pm)^2 + \frac{1}{4} \abs{\eta^\pm}^4.
    \end{equation*}
    Then we also use the fact that
    \begin{equation} \label{eq:link_beta_L}
        \Gamma^2 - \beta_* = \Gamma^2 (\cos^2 \theta_* - \sin^2 \theta_*),
    \end{equation}
    so that
    \begin{multline*}
        - \frac{1}{2} \int \beta_* \abs{\eta^\pm}^2 \psi_R^\pm (x) \diff x + \frac{1}{2} \int \Bigl( (- \partial^2_{xx} \nu^\pm + \Gamma^2 \nu^\pm) \nu^\pm + (- \partial_{xx}^2 \rho^\pm + \Gamma^2 \rho^\pm) \rho^\pm \Bigr) \psi_R^\pm (x) \diff x \\
            = \frac{1}{2} \Bigl( (L_\Gamma \nu^\pm, \nu^\pm)_{\psi_R^\pm} + (L_\Gamma \rho^\pm, \rho^\pm)_{\psi_R^\pm} \Bigr) - \frac{1}{8} \int \beta_*^\pm \abs{\eta^\pm}^4 \psi_R^\pm (x) \diff x.
    \end{multline*}
    Moreover, using \eqref{eq:est_phi_A}, we get
    \begin{align*}
        \abs{\int \eta^\pm \cdot \partial_x w_*^\pm \partial_x \psi_R^\pm (x) \diff x} & \leq \frac{C}{R} \norm{\eta^\pm}_{L^2 (\supp \partial_x \psi_R^\pm)} \norm{\partial_x w_*^\pm}_{L^2 (\supp \partial_x \psi_R^\pm)}, \\
        \abs{\int \eta^\pm \cdot (e_1 \wedge w_*^\pm) \partial_x \psi_R^\pm (x) \diff x} & \leq \frac{C}{R} \norm{\eta^\pm}_{L^2 (\supp \partial_x \psi_R^\pm)} \norm{e_1 + w_*^\pm}_{L^2 (\supp \partial_x \psi_R^\pm)}, \\
        \abs{\int \abs{\eta^\pm}^2 \partial^2_{xx} \psi_R^\pm (x) \diff x} & \leq \frac{C}{R^2} \norm{\eta^\pm}_{L^2 (\supp \partial_x \psi_R^\pm)}^2.
    \end{align*}
    Therefore, there holds
    \begin{align*}
     \MoveEqLeft  \abs{E^\pm - \Bigl[ \tilde{E}^\pm (w^\pm) + \frac{1}{2} \Bigl( (L_\Gamma \nu^\pm, \nu^\pm)_{\psi_R^\pm} + (L_\Gamma \rho^\pm, \rho^\pm)_{\psi_R^\pm} \Bigr) \Bigr] } \\ 
            & \leq C \Bigl[ \norm{\eta^\pm}_{H^1 (\supp \psi_R^\pm)}^3 + \frac{1}{R} \norm{\eta^\pm}_{L^2 (\supp \partial_x \psi_R^\pm)} \norm{\partial_x w_*^\pm}_{L^2 (\supp \partial_x \psi_R^\pm)} \\
            & \qquad + \frac{1}{R} \norm{\eta^\pm}_{L^2 (\supp \partial_x \psi_R^\pm)} \norm{e_1 + w_*^\pm}_{L^2 (\supp \partial_x \psi_R^\pm)} + \frac{1}{R^2} \norm{\eta^\pm}_{L^2 (\supp \partial_x \psi_R^\pm)}^2 \Bigl],
    \end{align*}
    and the conclusion comes from \eqref{eq:ineg2} and Lemma \ref{lem:est_w_infty}.
\end{proof}



However, in the previous lemma, $\tilde{E}^\pm (w_*^\pm)$ is not a constant : it still depends on the localisation $\psi_R^\pm$, and therefore on $y^\pm$. The following lemma estimates how far this quantity is from the constant $E (w_*) \coloneqq E (w_*^\pm)$.

\begin{lem} \label{lem:est_const}
    There exists $C > 0$ such that, for any $y^\pm$ such that $y^\pm - R > 0$ for $i=1$ and $2$, there holds
    \begin{equation*}
        \abs{\tilde{E}^\pm (w_*^\pm) - E (w_*)} \leq C \, e^{2 \Gamma (R - y^\pm)}.
    \end{equation*}
\end{lem}

\begin{proof}
    By the properties of $w^\pm$ and $\psi_R^\pm$, we have
    \begin{align*}
        \abs{\tilde{E}^\pm (w^\pm) - E (w_*)} &= \abs{\frac{1}{2} \int \Bigl( \abs{\partial_x w_*^\pm}^2 + 2 \gamma \partial_x w_*^\pm \cdot (e_1 \wedge w_*^\pm) + \sin^2 \theta_* \Bigr) (1-\psi_R^\pm (x)) \diff x}
            \leq C \norm{w_*^\pm + e_1}_{H^1 (I_\pm)}^2,
    \end{align*}
    where $I_+ \coloneqq (- \infty, R - y^+)$ and $I_- \coloneqq (-R - y^-, \infty)$,
    and the conclusion follows from Lemma \ref{lem:est_w_infty}.
\end{proof}


As for the quadratic terms in Proposition \ref{prop:energy_exp}, we can estimate them by applying Lemma \ref{lem:loc_schr_op} to $\nu^\pm$ and $\rho^\pm$.
Applying also Lemma \ref{lem:est_eta_eps}, the following estimates hold.

\begin{cor} \label{cor:sc_prod_loc}
    Under the assumptions of Proposition \ref{prop:energy_exp}, there holds
    \begin{align*}
        \biggl(L_\Gamma \nu^\pm, \nu^\pm\biggr)_{\psi_R^\pm}  &\geq 4 \lambda_0 \norm{\nu^\pm}_{H^1 (\psi_R^\pm (x) \diff x)}^2 - \frac{1}{\lambda_0} \Bigl( \int \sqrt{\psi_R^\pm} \nu^\pm \sin( \theta_* ) \diff x \Bigr)^2 - \frac{C}{R^2} \Bigl( \norm{\varepsilon}_{L^2}^2 + e^{2 \Gamma (R \pm y^\mp) } \Bigr), \\
        \biggl(L_\Gamma \rho^\pm, \rho^\pm\biggr)_{\psi_R^\pm} & \geq 4 \lambda_0 \norm{\rho^\pm}_{H^1 (\psi_R^\pm (x) \diff x)}^2 - \frac{1}{\lambda_0} \Bigl( \int \sqrt{\psi_R^\pm} \rho^\pm \sin( \theta_* ) \diff x \Bigr)^2 - \frac{C}{R^2} \Bigl( \norm{\varepsilon}_{L^2}^2 + e^{2 \Gamma (R \pm y^\mp) } \Bigr), \\
        \biggl(L_\Gamma \nu^\pm, \nu^\pm\biggr)_{\psi_R^\pm}  &\leq 2 \norm{\nu^\pm}_{H^1 (\psi_R^\pm (x) \diff x)}^2 + \frac{C}{R^2} \Bigl( \norm{\varepsilon}_{L^2}^2 + e^{2 \Gamma (R \pm y^\mp) } \Bigr), \\
        \biggl(L_\Gamma \rho^\pm, \rho^\pm\biggr)_{\psi_R^\pm} & \leq 2 \norm{\rho^\pm}_{H^1 (\psi_R^\pm (x) \diff x)}^2 + \frac{C}{R^2} \Bigl( \norm{\varepsilon}_{L^2}^2 + e^{2 \Gamma (R \pm y^\mp) } \Bigr),
    \end{align*}
\end{cor}


Putting everything together, we get bounds by below and above for $E^\pm (m)$.

\begin{cor} \label{cor:equiv_energy}
    Under the assumptions of Lemma \ref{lem:decomp_magn} and assuming $L_1 > R$, there exists $C > 0$ such that the following holds. With same notations as in the conclusion of Lemma \ref{lem:decomp_magn}, for all $t \in [0, T]$,
    \begin{multline*}
        2 C_2^2 \norm{\eta^\pm}_{H^1 (\psi_R^\pm \diff x)}^2 + C \Bigl( \norm{\varepsilon}_{H^1}^2 + e^{2 \Gamma (R - y^+)} + e^{2 \Gamma (R + y^-)} \Bigr) \geq E^\pm (m) - E (w_*) \\ \geq \frac{2 \lambda_0}{C^2} \norm{\eta^\pm}_{H^1 (\psi_R^\pm \diff x)}^2 - C \Bigl( \norm{\varepsilon}_{H^1}^3 + \frac{1}{R^2} \norm{\varepsilon}_{H^1}^2 + e^{- 2 \Gamma (R - y^+)} + e^{2 \Gamma (R + y^-)} \Bigr).
    \end{multline*}
\end{cor}

\begin{proof}
    We use Proposition \ref{prop:energy_exp} whose assumptions are satisfied thanks to the conclusion of Lemma \ref{lem:decomp_magn}. Thus, we get for all $t \in [0, T]$
\[
        E^\pm (m) - \tilde{E}^\pm (w_*^\pm)
        \geq \frac{1}{2} \Bigl( (L_\Gamma \nu^\pm, \nu^\pm)_{\psi_R^\pm} + (L_\Gamma \rho^\pm, \rho^\pm)_{\psi_R^\pm} \Bigr) \Bigr] - C \Bigl[ \norm{\varepsilon}_{H^1}^3 + \frac{1}{R^2} \norm{\varepsilon}_{L^2}^2 + e^{2 \Gamma (R - y^+)} + e^{2\Gamma (R + y^-)} \Bigl],
\]
    From this inequality, we can substitute $\tilde{E}^\pm (w_*^\pm)$ into $E (w_*)$ thanks to Lemma \ref{lem:est_const}. Then, both $(L_\Gamma \nu^\pm, \nu^\pm)_{\psi_R^\pm}$ and $(L_\Gamma \rho^\pm, \rho^\pm)_{\psi_R^\pm}$ can be estimated by below by using Corollary \ref{cor:sc_prod_loc}. Moreover, the terms
    \begin{equation*}
        \Bigl( \int \sqrt{\psi_R^\pm} \nu^\pm \sin( \theta_* ) \diff x \Bigr)^2
        \qquad \text{and} \qquad
        \Bigl( \int \sqrt{\psi_R^\pm} \rho^\pm \sin( \theta_* ) \diff x \Bigr)^2
    \end{equation*}
    are controlled by the "almost-orthogonality" estimates of Lemma \ref{lem:almst_orth}. From this estimate, we have
    \begin{equation*}
        q(y^+ - y^-)^2 \leq C e^{-3 \Gamma (y^+ - y^-)/2} \leq C e^{- 2 \Gamma y^+} + C e^{2 \Gamma y^-},
    \end{equation*}
    and thus the conclusion. The estimate by above is obtained with similar computations.
\end{proof}

Proposition \ref{prop:equiv_energy} follows by taking the sum of the two estimates of Corollary \ref{cor:equiv_energy} and applying \eqref{eq:equiv_eta_eps} from Lemma \ref{lem:est_eta_eps}.

 \section{Evolution of the energy} \label{sec:en_diss}

In this section, we prove Proposition \ref{prop:est_dt_en}. The evolution of the energy is already known from \cite{Cote_Ignat__stab_DW_LLG_DM}. We recall it here.

\begin{lem}[{\cite[Theorem~4.1]{Cote_Ignat__stab_DW_LLG_DM}}]
Define the \emph{dissipation term}
\index{Functionals!$D$: dissipation term}
\begin{equation*}
    D(t) \coloneqq \int (\abs{\delta E (m)}^2 - \abs{m \cdot \delta E (m)}^2 ) \diff x,
\end{equation*}
and the \emph{forcing term}
\begin{equation*}
\index{Functionals!$F$: force term related to the field $h$}
    F(t) \coloneqq \int (m \wedge e_1) \cdot (m \wedge \delta E (m)) \diff x.
\end{equation*}
Then there hold the energy dissipation equality:
\begin{equation} \label{eq:en_diss}
    \frac{\diff}{\diff t} E(m(t)) = - \alpha D(t) + \alpha h(t) F(t) .
\end{equation}
\end{lem}

\subsection{Localisation}

Proceeding as before, we localise each of these terms :
\begin{gather*} 
\index{Functionals!$D^\pm$: localized dissipation term, with transition around $0$}
\index{Functionals!$F^\pm$: localized force term, with transition around $0$}
    D^\pm  (m) \coloneqq \int (\abs{\delta E (m)}^2 - \abs{m \cdot \delta E (m)}^2 ) \psi_R (\pm x) \diff x, \\
    F^\pm  (m) \coloneqq \int (m \wedge e_1) \cdot (m \wedge \delta E (m)) \psi_R (\pm x) \diff x, \\
\end{gather*}
Therefore, we have $F = F^+ + F^-$ and $D = D^+ + D^-$.
We also define, in a similar way as previously,
\begin{gather*}
\index{Functionals!$\tilde D^\pm$: localized dissipation term, with transition not far from a domain wall}
\index{Functionals!$\tilde F^\pm$: localized force term, with transition not far from a domain wall}
    \tilde{D}^\pm (m) \coloneqq D^\pm (\tau_{y^\pm} m) = \int (\abs{\delta E (m)}^2 - \abs{m \cdot \delta E (m)}^2 ) \psi_R^\pm (x) \diff x, \\
    \tilde{F}^\pm (m) \coloneqq F^\pm (\tau_{y^+} m) = \int (m \wedge e_1) \cdot (m \wedge \delta E (m)) \psi_R^\pm (x) \diff x, \\
\end{gather*}

\subsection{Estimates on the localised terms}

\subsubsection{Dissipation term}

First, we show that the dissipation term is positive and can be estimated up to some negligible terms.

\begin{lem}
    Under the assumptions of Proposition \ref{prop:energy_exp}, there exists $C > 0$ such that, if $\norm{\varepsilon}_{H^1} \leq 1$,
    \begin{equation} \label{eq:first_est_diss}
        \abs{D^\pm (m) - \Bigl( \norm{L_\Gamma \nu^\pm}_{L^2 (\psi_R^\pm \diff x)}^2 + \norm{L_\Gamma \rho^\pm}_{L^2 (\psi_R^\pm \diff x)}^2 \Bigr)} \leq C \Bigl( \norm{\varepsilon}_{H^2}^2 (\norm{\varepsilon}_{H^1} + e^{\Gamma (R \pm y^\mp)}) + e^{3 \Gamma (R \pm y^\mp)} \Bigr).
    \end{equation}
\end{lem}

\begin{proof}
    First, we recall that
    \begin{equation*}
        m = g^+.(\eta^+ + w^+_*) = g^-.(\eta^- + w_*^-),
    \end{equation*}
    so that
    \begin{equation*}
        D^\pm (m) = \tilde{D}^\pm (\eta^\pm + w^\pm_*)
    \end{equation*}
    Once again, the pointwise estimate of the steps 3 and 5 of the proof of \cite[Proposition~4.16]{Cote_Ignat__stab_DW_LLG_DM} still holds here, so that we get:
    \begin{gather*}
        \abs{\delta E (\eta^\pm + w^\pm_*)}^2 - \abs{\delta E (w_*^\pm)}^2 = 2 \beta_* w_*^\pm \cdot \delta E (\eta^\pm) + \abs{\delta E (\eta^\pm)}^2, \\
        \begin{aligned}
        \abs{(\eta^\pm + w_*^\pm) \cdot \delta E (\eta^\pm + w_*^\pm)}^2 - \abs{w_*^\pm \cdot \delta E (w_*^\pm)}^2 & = 2 \beta_*^2 \mu^\pm + 2 \beta_* w_*^\pm \cdot \delta E (\eta^\pm) + 2 \beta_* \eta^\pm \cdot \delta E (\eta^\pm) \\
        & \qquad + \abs{w_*^\pm \cdot \delta E (\eta^\pm)}^2 + O_4^3 (\eta^\pm) + O_4^4 (\eta^\pm), 
        \end{aligned} \\
        \abs{w_*^\pm \cdot \delta E (\eta^\pm)}^2 = (2 \partial_{xx}^2 \theta_* \nu^\pm + 2 \partial_x \theta_* \partial_x \nu^\pm)^2 + O_4^3 (\eta^\pm) + O_4^4 (\eta^\pm), \\
        \abs{\delta E (\eta^\pm)}^2 = (2 \partial^2_{xx} \theta_* \nu^\pm + 2 \partial_x \theta_* \partial_x \nu^\pm)^2 + (- \partial_{xx}^2 \nu^\pm + \Gamma^2 \nu^\pm)^2 + (- \partial_{xx}^2 \rho^\pm + \Gamma^2 \rho^\pm)^2 + O_4^3 (\eta^\pm) + O_4^4 (\eta^\pm).
    \end{gather*}
    We also recall that $\mu^\pm = - \frac{1}{2} ((\nu^\pm)^2 + (\rho^\pm)^2) - \frac{1}{8} \abs{\eta^\pm}^4$ and that $\abs{\delta E (w_*^\pm)} = {w_*^\pm \cdot \delta E (w_*^\pm)} = \beta_*$. Therefore:
    \begin{equation*}
        D^\pm (m) = \int \Bigl[ ( - \partial_{xx} \nu^\pm + \Gamma^2 \nu^\pm - \beta_* \nu^\pm )^2 + ( - \partial_{xx} \rho^\pm + \Gamma^2 \rho^\pm - \beta_* \rho^\pm )^2 + O_4^3 (\eta^\pm) + O_4^4 (\eta^\pm) \Bigr] \psi_R^\pm (x) \diff x.
    \end{equation*}
    Using \eqref{eq:link_beta_L} and Lemma \ref{lem:tech_o_lem}, we get
    \begin{multline*}
        D^\pm (m) - \int \Bigl[ \abs{L_\Gamma \nu^\pm}^2 + \abs{L_\Gamma \rho^\pm}^2 \Bigr] \psi_R^\pm (x) \diff x \\ = O \Bigl( \norm{\eta^\pm}_{H^1 (\supp \psi_R^\pm)} \norm{\eta^\pm}_{H^2 (\supp \psi_R^\pm)}^2 \Bigr) + O \Bigl( \norm{\eta^\pm}_{H^1 (\supp \psi_R^\pm)}^2 \norm{\eta^\pm}_{H^2 (\supp \psi_R^\pm)}^2 \Bigr).
    \end{multline*}
    The conclusion follows from \eqref{eq:ineg2} and the fact that both $\norm{\varepsilon}_{H^1} \leq 1$ and $e^{\Gamma (R - y^\pm)} \leq 1$.
\end{proof}

Once again, the localisation remains in the quadratic terms of the left-hand side. 
Applying Lemma \ref{lem:loc_schr_op} to both $\nu^\pm$ and $\rho^\pm$, along with Lemma \ref{lem:est_eta_eps}, we obtain the following estimates:

\begin{cor} \label{cor:est_F1}
    Under the assumptions of Proposition \ref{prop:energy_exp}, there exists $C, \lambda_5 > 0$ such that, for $R \geq 1$,
    \begin{multline*}
        D^\pm \geq \lambda_5 \norm{\eta^\pm}_{H^2 (\psi_R^\pm \diff x)}^2 - C \Bigl( \norm{\varepsilon}_{H^2}^2 (\norm{\varepsilon}_{H^1} + e^{\Gamma (R \pm y^\mp)} + \frac{1}{R}) + e^{2 \Gamma (R \pm y^\mp)} \Bigr) \\ - \frac{1}{\lambda_0} \Bigl( \int \sqrt{\psi_R^\pm} \nu^\pm \sin (\theta_*) \diff x \Bigr)^2 - \frac{1}{\lambda_0} \Bigl( \int \sqrt{\psi_R^\pm} \rho^\pm \sin (\theta_*) \diff x \Bigr)^2,
    \end{multline*}
    \begin{multline}
        D \geq \lambda_5 \norm{\varepsilon}_{H^2}^2 - C \Bigl( \norm{\varepsilon}_{H^2}^2 (\norm{\varepsilon}_{H^1} + e^{\Gamma (R - y^+)} + e^{\Gamma (R + y^-)} + \frac{1}{R}) + e^{2 \Gamma (R - y^+)} + e^{2 \Gamma (R + y^-)} \Bigr) \\ - \frac{1}{\lambda_0} \sum_\iota \Bigl( \int \sqrt{\psi_R^\iota} \nu^\iota \sin (\theta_*) \diff x \Bigr)^2 + \Bigl( \int \sqrt{\psi_R^\iota} \rho^\iota \sin (\theta_*) \diff x \Bigr)^2.
    \end{multline}
\end{cor}


\subsubsection{Forcing term}

In a second step, we show that the forcing term is negligible enough with respect to the dissipation term.

\begin{prop}
    Under the assumptions of Proposition \ref{prop:energy_exp}, there exists $C > 0$ such that, if $\norm{\varepsilon}_{H^1} \leq 1$,
    \begin{equation*}
        \abs{F^\pm (m)} \leq C (\norm{\varepsilon}_{H^1}^2 + e^{2 \Gamma (R \pm y^\mp)}).
    \end{equation*}
\end{prop}

\begin{proof}
    Similarly, we have
    \begin{equation*}
        F^\pm (m) = \tilde{F}^\pm (\eta^\pm + w^\pm_*).
    \end{equation*}
    Again, we will take advantage of the pointwise computations of \cite{Cote_Ignat__stab_DW_LLG_DM}.
    \begin{multline*}
        ((\eta^\pm + w^\pm_*) \wedge e_1) \cdot ((\eta^\pm + w^\pm_*) \wedge \delta E (\eta^\pm + w^\pm_*)) = (w_*^\pm \wedge e_1) \cdot (w_*^\pm \wedge \delta E (w_*^\pm)) + (\eta^\pm \wedge e_1) \cdot (w_*^\pm \wedge \delta E (w_*^\pm)) \\ + (w_*^\pm \wedge e_1) \cdot (\eta^\pm \wedge \delta E (w_*^\pm) + w_*^\pm \wedge \delta E (\eta^\pm)) + O_2^2 (\eta^\pm) + O_2^3 (\eta^\pm).
    \end{multline*}
    Since $w_*^\pm \wedge \delta E (w_*^\pm) = 0$, the first two terms are 0. Then, observe that $w_*^\pm \wedge e_1 = \sin (\theta_*) p_*^\pm$ and $n_*^\pm \wedge w_*^\pm = - p_*^\pm$, thus
    \begin{align*}
        (w_*^\pm \wedge e_1) \cdot (\eta^\pm \wedge \delta E (w_*^\pm)) &= \sin (\theta_*) p_*^\pm \cdot (\eta^\pm \wedge \beta_* w_*^\pm) = - \beta_* \sin \theta_* \nu^\pm, \\
        (w_*^\pm \wedge e_1) \cdot (w_*^\pm \wedge \delta E (\eta^\pm)) &= \sin \theta_* (- \partial_{xx}^2 \nu^\pm + \Gamma^2 \nu^\pm) + O_2^2(\eta^\pm).
    \end{align*}
    Using \eqref{eq:link_beta_L}, we obtain
    \begin{equation*}
        ((\eta^\pm + w^\pm_*) \wedge e_1) \cdot ((\eta^\pm + w^\pm_*) \wedge \delta E (\eta^\pm + w^\pm_*)) = \sin \theta_* (L_\Gamma \nu^\pm) + O_2^2(\eta^\pm) + O_2^3 (\eta^\pm).
    \end{equation*}
    Moreover, we know that $L_\Gamma (\sin \theta_*) = 0$, therefore
    \begin{equation*}
        \int \sin \theta_* L_\Gamma \nu^\pm \psi_R^\pm (x) \diff x = 2 \int \partial_x (\sin \theta_*) \nu^\pm \partial_x \psi_R^\pm (x) \diff x + \int \sin \theta_* \nu^\pm \partial^2_{xx} \psi_R^\pm (x) \diff x.
    \end{equation*}
    Estimating these terms by using the fact that $\supp \partial_x \psi_R^\pm \subset [- y^\pm - R, - y^\pm + R]$ (and the same for $\partial^2_{xx} \psi_R^\pm$), we get
    \begin{align*}
        \abs{\int \partial_x (\sin \theta_*) \nu^\pm \partial_x \psi_R^\pm (x) \diff x} &\leq \frac{\norm{\partial_x \psi}_{L^\infty}}{R} \norm{\partial_x \sin \theta_*}_{L^2 ((- y^\pm - R, - y^\pm + R))} \norm{\nu^\pm}_{L^2 ((- y^\pm - R, - y^\pm + R))} \\ &\leq C \frac{\norm{\partial_x \psi}_{L^\infty}}{R} e^{\Gamma (R \pm y^\mp)} \norm{\nu^\pm}_{L^2 ((- y^\pm - R, - y^\pm + R))}, \\
        \abs{\int \sin \theta_* \nu^\pm \partial^2_{xx} \psi_R^\pm (x) \diff x} &\leq \frac{\norm{\partial^2_{xx} \psi}_{L^\infty}}{R^2} \norm{\sin \theta_*}_{L^2 ((- y^\pm - R, - y^\pm + R))} \norm{\nu^\pm}_{H^1 ((- y^\pm - R, - y^\pm + R))} \\ &\leq C \frac{\norm{\partial^2_{xx} \psi}_{L^\infty}}{R^2} e^{\Gamma (R \pm y^\mp)} \norm{\nu^\pm}_{H^1 ((- y^\pm - R, - y^\pm + R))}.
    \end{align*}
    Thus, we obtain
    \begin{equation*}
        \abs{\int \sin \theta_* L_\Gamma \nu^\pm \psi_R^\pm (x) \diff x} \leq \frac{C}{R^2} \Bigl( \norm{\nu^\pm}_{H^1 ((- y^\pm - R, - y^\pm + R))}^2 + e^{2 \Gamma (R \pm y^\mp)} \Bigr),
    \end{equation*}
    and the conclusion follows in the same way, using again Lemma \ref{lem:tech_o_lem}.
\end{proof}

By summing for $\iota = \pm 1$, we get an estimate for $F$.

\begin{cor} \label{cor:est_F2}
    Under the assumptions of Proposition \ref{prop:energy_exp}, there exists $C > 0$ such that, if $\norm{\varepsilon}_{H^1} \leq 1$,
    \begin{equation}
        \abs{F (m)} \leq C (\norm{\varepsilon}_{H^1}^2 + e^{2 \Gamma (R - y^+)} + e^{2 \Gamma (R + y^-)})
    \end{equation}
\end{cor}

\subsection{Dissipation estimate}

Now, we prove Proposition \ref{prop:est_dt_en} thanks to the previous lemmas.

    Recall the energy dissipation equation \eqref{eq:en_diss}. From the estimates of Corollaries \ref{cor:est_F1} and \ref{cor:est_F2} on $D$ and $F$ respectively, we get
    \begin{multline*}
         \frac{\diff}{\diff t} E (m) + \lambda_5 \norm{\varepsilon}_{H^2}^2  \leq C \alpha \norm{\varepsilon}_{H^2}^2 \Bigl( \norm{\varepsilon}_{H^1} + e^{\Gamma (R - y^+)} + e^{\Gamma (R + y^-)} + \frac{1}{R} \Bigr) + \abs{h(t)} \norm{\varepsilon}_{H^1}^2 \\
                 + C \alpha \abs{h(t)} \Bigl( e^{2 \Gamma (R - y^+)} + e^{2 \Gamma (R + y^-)} \Bigr)+ \frac{\alpha}{\lambda_0} \sum_{i=1}^2 \Bigl( \int \sqrt{\psi_R^\pm} \nu^\pm \sin (\theta_*) \diff x \Bigr)^2 + \Bigl( \int \sqrt{\psi_R^\pm} \rho^\pm \sin (\theta_*) \diff x \Bigr)^2.
    \end{multline*}
    The conclusion follows by applying Lemma \ref{lem:almst_orth} to the last terms of the right-hand side.

\section{Proof of the decomposition of the magnetization} \label{sec:decomp}

In this section, we prove Lemma \ref{lem:decomp_magn} which decomposes the magnetization with two nice gauges. This result is a consequence of the implicit function theorem; in order to take into account the constraint that the magnetization is $S^2$-valued, it  is used here in the context of Banach manifolds.

For $g^\pm \in G$ such that $\pm y^\pm \geq 0$, we define the profile 
\[ \index{Functions!$P_{g^+, g^-}$: two domain wall profile} P_{g^+, g^-} \coloneqq g^+ . w_*^+ + g^- . w_*^- + e_1. \]
We also call $P_{0, g^-} \coloneqq P_{(0,0), g^-}$.

\subsection{Definition and properties of the functional}

We start with defining and studying the appropriate functional: let
\begin{equation*}
\index{Functionals!$\q F, \overline{\q F}$: related to modulation}
    \overline{\mathcal{F}} : 
    \begin{aligned}[t]
        (L^2 + L^\infty) \times G^2 &\rightarrow \mathbb{R}^4 \\
        (m, g^+, g^-) &\mapsto 
            \begin{pmatrix}
                \int m \cdot (g^+ . \partial_x w_*^+) \diff x \\
                \int m \cdot (e_1 \wedge (g^+ . w_*^+)) \diff x \\
                \int m \cdot (g^- . \partial_x w_*^-) \diff x \\
                \int m \cdot (e_1 \wedge (g^- .  w_*^-)) \diff x
            \end{pmatrix}.
        \end{aligned}
\end{equation*}
and
\begin{equation*}
    \mathcal{F} : 
    \begin{aligned}[t]
        \mathcal{H}^1 \times G^2 &\rightarrow \mathbb{R}^4 \\
        (m, g^+, g^-) &\mapsto 
            \begin{pmatrix}
                \int \varepsilon \cdot (g^+ . \partial_x w_*^+) \diff x \\
                \int \varepsilon \cdot (e_1 \wedge (g^+ . w_*^+)) \diff x \\
                \int \varepsilon \cdot (g^- . \partial_x w_*^-) \diff x \\
                \int \varepsilon \cdot (e_1 \wedge (g^- .  w_*^-)) \diff x
            \end{pmatrix},
        \end{aligned}
\end{equation*}
where $\varepsilon \coloneqq m - P_{g^+, g^-} \in H^1$, so that $\mathcal{F} (m, g^+, g^-) = \overline{\mathcal{F}} (m - P_{g^+, g^-}, g^+, g^-)$. Remark also that $\overline{\mathcal{F}} ( \cdot , g^+, g^-)$ is linear.
%
%

\begin{prop} \label{prop:lin_est_F}
    There exists $C > 0$ such that there holds for all $m \in X$ and all $g^\pm \in G$
    \begin{equation*}
        \abs{\overline{\mathcal{F}} (m, g^+, g^-)} \leq C \norm{m}_{X},
    \end{equation*}
    for $X = L^2$ or $L^\infty$.
\end{prop}

\begin{proof}
   The result easily follows the fact that both $\partial_x w_*^{\pm}$ and $e_1 \wedge w_*^{\pm}$ are bounded and decay exponentially at infinity due to Lemma \ref{lem:est_w_infty}.
\end{proof}

\begin{cor}
    There exists $C > 0$ such that there holds for all $m, m' \in \mathcal{H}^1$ and all $g^\pm \in G$
    \begin{equation*}
        \abs{\mathcal{F} (m, g^+, g^-) - \mathcal{F} (m', g^+, g^-)} \leq C \norm{m - m'}_{X},
    \end{equation*}
    \begin{equation*}
        \abs{\mathcal{F} (m, g^+, g^-)} \leq C \norm{m - {P}_{g^+, g^-}}_{X},
    \end{equation*}
    for $X = L^2$ or $L^\infty$.
\end{cor}

\begin{proof}
    Those estimates can be easily deduced from Proposition \ref{prop:lin_est_F} and the definition of $\mathcal{F}$, which gives in particular 
    \[ \mathcal{F} (m, g^+, g^-) - \mathcal{F} (m', g^+, g^-) = \overline{\mathcal{F}} (m - m', g^+, g^-). \qedhere \]
\end{proof}

Similarly, we prove a similar property for the partial differential $D_{g^+, g^-} \overline{\mathcal{F}}$, which is represented by a $4 \times 4$ matrix, endowed with a norm denoted $\| \cdot \|$.

\begin{lem} \label{lem:lin_est_DF}
    There exists $C > 0$ such that there holds for all $m \in X$ and all $g^\pm \in G$
    \begin{equation*}
        \norm{D_{g^+, g^-} \overline{\mathcal{F}} (m, g^+, g^-)} \leq C \norm{m}_{X},
    \end{equation*}
    for $X = L^2$ or $L^\infty$.
\end{lem}

\begin{proof}
    From the definition of $\overline{\mathcal{F}}$, we see that
    \begin{align*}
        \partial_{y^+} \overline{\mathcal{F}} (m, g^+, g^-) &= 
            \begin{pmatrix}
                - \int m \cdot (g^+ . \partial^2_{xx} w_*^+) \diff x \\
                - \int m \cdot (e_1 \wedge (g^+ . \partial_x w_*^+)) \diff x \\
                0 \\
                0
            \end{pmatrix} \quad \text{and} \quad
        \partial_{\phi^+} \overline{\mathcal{F}} (m, g^+, g^-) &= 
            \begin{pmatrix}
                \int m \cdot (e_1 \wedge (g^+ . \partial_x w_*^+)) \diff x \\
                \int m \cdot (e_1 \wedge (e_1 \wedge (g^+ . w_*^+))) \diff x \\
                0 \\
                0
            \end{pmatrix}.
    \end{align*}
    Therefore, since $\partial^2_x w_*^+$, $e_1 \wedge \partial_x w_*^+$ and $e_1 \wedge w_*^+$ decay exponentially at infinity thanks to Lemma \ref{lem:est_w_infty}, we get the conclusion for these two differentials. As for $\partial_{y^-} \overline{\mathcal{F}}$ and $\partial_{\phi^-} \overline{\mathcal{F}}$, the same arguments give the conclusion.
\end{proof}

Let also define $\mathcal{F}_0 (\zeta^+, \zeta^-, g^+, g^-) \coloneqq \overline{\mathcal{F}} (P_{\zeta^+, \zeta^-}, g^+, g^-)$\index{Functions!$\q F_0$: functional $\q F$ evaluated at a given 2 domain wall profile $P$ and gauge $g$} for any $\zeta^\iota, g^\iota \in G$, and

\begin{equation*}
\index{Constants!$A$: leading order matrix of the differential of $\q F$ with respect to the gauges}
    A \coloneqq \frac{2}{\Gamma} 
            \begin{pmatrix}
                1 & \gamma & 0 & 0 \\
                - \gamma & -1 & 0 & 0 \\
                0 & 0 & 1 & \gamma \\
                0 & 0 & - \gamma & -1
            \end{pmatrix}
\end{equation*}
We point out that $A$ is invertible as soon as $\gamma^2 < 1$.

\begin{lem} \label{lem:est_DF_P}
    For all $g^\pm \in G$, there holds
    \begin{equation*}
        \norm{D_{\zeta^+, \zeta^-} \mathcal{F}_0 (g^+, g^-, g^+, g^-) + A} \leq C q(y^+ - y^-).
    \end{equation*}
\end{lem}

\begin{proof}
    From the definition of $\mathcal{F}_0$ and by noting $\zeta^\iota = (z^\iota, \alpha^\iota)$, we know that
    \begin{equation*}
        \partial_{z^+} \mathcal{F}_0 (\zeta^+, \zeta^-, g^+, g^-) = \overline{\mathcal{F}} ( \partial_{z^+} (P_{\zeta^+, \zeta^-}), g^+, g^-) = - \overline{\mathcal{F}} ( \zeta^+ . \partial_x w_*^+, g^+, g^- ).
    \end{equation*}
    Therefore, taking $\zeta^+ = g^+$ and applying Lemma \ref{lem:int_w}, we get
    \begin{equation*}
        \partial_{z^+} \mathcal{F}_0 (g^+, g^-, g^+, g^-) = - \frac{2}{\Gamma} 
            \begin{pmatrix}
                1 \\
                - \gamma \\
                0 \\
                0 
            \end{pmatrix}
            -
            \begin{pmatrix}
                0 \\
                0 \\
                \int g^+ . \partial_x w_*^+ (-x) \cdot (g^- . \partial_x w_*^-) \diff x \\
                \int g^+ . \partial_x w_*^+ (-x) \cdot (e_1 \wedge (g^- .  w_*^-)) \diff x
            \end{pmatrix}.
    \end{equation*}
    The last term can be estimated thanks to Corollary \ref{cor:est_w_orthogonal2}.
    Then, there also holds
    \begin{equation*}
        \partial_{\alpha^+} \mathcal{F}_0 (\zeta^+, \zeta^-, g^+, g^-) = \overline{\mathcal{F}} ( \partial_{\phi^+} (P_{\zeta^+, \zeta^-}), g^+, g^-) = \overline{\mathcal{F}} ( e_1 \wedge \zeta^+ . w_*^+, g^+, g^- ).
    \end{equation*}
    Thus, applying Lemma \ref{lem:int_w} again,
    \begin{equation*}
        \partial_{\alpha^+} \mathcal{F}_0 (g^+, g^-, g^+, g^-) = - \frac{2}{\Gamma} 
            \begin{pmatrix}
                \gamma \\
                - 1 \\
                0 \\
                0 
            \end{pmatrix}
            -
            \begin{pmatrix}
                0 \\
                0 \\
                \int e_1 \wedge g^+ . w_*^+ (-x) \cdot (g^- . \partial_x w_*^-) \diff x \\
                \int e_1 \wedge g^+ . w_*^+ (-x) \cdot (e_1 \wedge (g^- .  w_*^-)) \diff x
            \end{pmatrix},
    \end{equation*}
    and the last term can be estimated again with Corollary \ref{cor:est_w_orthogonal2}. 
    
    Similar computations for $\partial_{z^-} \mathcal{F}_0 (g^+, g^-, g^+, g^-)$ and $\partial_{\phi^-} \mathcal{F}_0 (g^+, g^-, g^+, g^-)$ give the conclusion.
\end{proof}

\begin{lem} \label{lem:est_DF_A}
    For any $m \in \mathcal{H}^1$ and any $g^\pm \in G$, there holds
    \begin{equation*}
        \norm{D_{g^+, g^-} \mathcal{F} (m, g^+, g^-) - A} \leq C \Bigl( \norm{m - {P}_{g^+, g^-}}_{X} + q(y^+ - y^-) \Bigr),
    \end{equation*}
    for $X = L^2$ or $L^\infty$.
\end{lem}

\begin{proof}
    From the definition of $\mathcal{F}$ and $\overline{\mathcal{F}}$, there holds
    \begin{equation*}
        \mathcal{F} (m, g^+, g^-) = \overline{\mathcal{F}} (m, g^+, g^-) - \overline{\mathcal{F}} (P_{g^+, g^-}, g^+, g^-).
    \end{equation*}
    Therefore, using the fact that $\overline{\mathcal{F}} ( \cdot , g^+, g^-)$ is linear and then so is $D_{g^+, g^-} \overline{\mathcal{F}} ( \cdot , g^+, g^-)$, we obtain
    \begin{align*}
        D_{g^+, g^-} \mathcal{F} (m, g^+, g^-) &= D_{g^+, g^-} \overline{\mathcal{F}} ( m, g^+, g^-) - D_{g^+, g^-} \overline{\mathcal{F}} ( P_{g^+, g^-}, g^+, g^-) - D_{\zeta^+, \zeta^-} \mathcal{F}_0 (g^+, g^-, g^+, g^-) \\
            &= D_{g^+, g^-} \overline{\mathcal{F}} ( m - P_{g^+, g^-}, g^+, g^-) - D_{\zeta^+, \zeta^-} \mathcal{F}_0 (g^+, g^-, g^+, g^-).
    \end{align*}
    The conclusion is reached by applying Lemma \ref{lem:lin_est_DF} to the first term of the right-hand side and Lemma \ref{lem:est_DF_P} to the second term.
\end{proof}

\subsection{Construction of the gauges}

We are now in a position to apply the implicit function theorem, except for lipschitz type estimate on $\mathcal F$. We will need an intermediate result for $P_{g^+, g^-}$ in the same context as Lemma \ref{lem:est_g_w}, as follows.

\begin{lem} \label{lem:diff_double_P_tilde}
    For all $g^{[1]}, g^{[2]}, g^{[3]}, g^{[4]} \in G$, there holds
    \begin{equation*}
        \norm{{P}_{g^{[1]}, g^{[2]}} - {P}_{g^{[3]}, g^{[4]}}}_{H^1} \leq C \Bigl( \abs{g^{[1]} - g^{[3]}} + \abs{g^{[2]} - g^{[4]}} \Bigr).
    \end{equation*}
    Moreover, let $y_0 \ge 0$ be such that $(-1)^i y^{[i]} \ge y_0$ for $i \in \{ 1,2,3,4 \}$. Then there also holds
    \begin{equation*}
        \norm{g^{[1]} . w_*^+ - g^{[3]} . w_*^+}_{H^1}^2 + \norm{g^{[2]} . w_*^- - g^{[4]} . w_*^-}_{H^1}^2 \leq \norm{{P}_{g^{[1]}, g^{[2]}} - {P}_{g^{[3]}, g^{[4]}}}_{H^1}^2 + C q(2y_0).
    \end{equation*}
\end{lem}

\begin{proof}
    By using Lemma \ref{lem:est_g_w} and the fact that
    \begin{equation*}
        P_{g^{[1]}, g^{[2]}} - P_{g^{[3]}, g^{[4]}} = (g^{[1]}.w_*^+ - g^{[3]}.w_*^+) + (g^{[2]}.w_*^- - g^{[4]}.w_*^-),
    \end{equation*}
    we get the first estimate. As for the second one,
    we expand
    \begin{multline*}
        \norm{P_{g^{[1]}, g^{[2]}} - P_{g^{[3]}, g^{[4]}}}_{H^1}^2 = \norm{g^{[1]}.w_*^+ - g^{[3]}.w_*^+}_{H^1}^2 + \norm{g^{[2]}.w_*^- - g^{[4]}.w_*^-}_{H^1}^2 \\ + 2 \langle g^{[1]}.w_*^+ - g^{[3]}.w_*^+, g^{[2]}.w_*^- - g^{[4]}.w_*^- \rangle_{H^1}.
    \end{multline*}
    From Lemma \ref{lem:est_g_w}, assuming for instance $y^{[1]} \leq y^{[3]}$ and $y^{[4]} \leq y^{[2]}$, we have for $j \in \{ 0, 1 \}$
    \begin{equation*}
        \abs{\partial_x^j g^{[1]}.w_*^+ (x) - \partial_x^j g^{[3]}w_*^+ (x)} \leq C e^{- \Gamma \max{(0, x - y^{[3]}, y^{[1]}-x)}},
    \end{equation*}
    \begin{equation*}
        \abs{\partial_x^j g^{[2]}.w_*^- (x) - \partial_x^j g^{[4]}w_*^- (x)} \leq C e^{- \Gamma \max{(0, x - y^{[2]}, y^{[4]}-x)}}.
    \end{equation*}
    Then, a computation as in Corollary \ref{cor:est_w_orthogonal} shows that 
    
    \begin{equation*}
        \abs{\partial_x^j g^{[1]}.w_*^+ (x) - \partial_x^j g^{[3]}w_*^+ (x)} \abs{\partial_x^j g^{[2]}.w_*^- (-x) - \partial_x^j g^{[4]}w_*^- (-x)} \leq C e^{- 2 \Gamma y_0 } \times
        \begin{cases}
            e^{-\Gamma (x-y_0)} &\qquad \text{if } x \geq y_0 \\
            1 &\qquad \text{if } x \in [- y_0,y_0] \\
            e^{\Gamma (x - y_0)} &\qquad \text{if } x \leq - y_0
        \end{cases}.
    \end{equation*}
    Therefore, we obtain as in Corollary \ref{cor:est_w_orthogonal2}
    \begin{equation*}
        \abs{\langle g^{[1]}.w_*^+ - g^{[3]}.w_*^+, (g^{[2]}.w_*^- - g^{[4]}.w_*^-) (-x) \rangle_{H^1}} \leq C q(2y_0),
    \end{equation*}
    and the conclusion follows.
\end{proof}

\begin{lem} \label{lem:static_implicit}
    \begin{enumerate}
        \item There exists $C > 0$, $L_1' \ge 1$ and $\delta_1' > 0$ such that, for all $m \in \mathcal{H}^1$ and $g_0 \in G$ satisfying $y_0 < - L_1'$ and
        \begin{equation*}
            \delta \coloneqq \norm{m - P_{0, g_0}}_{H^1} < \delta_1',
        \end{equation*}
        there exist unique $\overline{g}^\pm \in G$ such that
        \begin{itemize}
            \item $\abs{\overline{g}^- - g_0} \leq C \delta$ and $\abs{\overline{g}^+} \leq C \delta$,
            \item $\norm{m - P_{\overline{g}^+, \overline{g}^-}}_{H^1} \leq C \delta$,
            \item $\mathcal{F} (m, \overline{g}^+, \overline{g}^-) = 0$.
        \end{itemize}
        Moreover, 
        $(\overline{g}^+, \overline{g}^-)$ does not depend on $g_0$.
        \item The map
        \begin{align*}
            \{ m \in e_1 + H^1 | \inf_{y_0 < - L_1'} \norm{m - P_{0, g_0}}_{H^1} < \delta_0 \} &\rightarrow G^2 \\
            m &\mapsto (\overline{g}^+, \overline{g}^-) \text{ as in  1)}
        \end{align*}
        is $\mathscr{C}^\infty$ with respect to the $H^1$ topology.
    \end{enumerate}
\end{lem}

\begin{proof}
    \emph{Step 1. Existence and uniqueness}
    First, remark that $\mathcal{F} (m, \overline{g}^+, \overline{g}^-) = 0$ is equivalent to $\overline{p} = \overline{p} - A^{-1} \mathcal{F} (m, \overline{p})$ where $\overline{p} = (\overline{g}^+, \overline{g}^-)$. We define $\mathcal{G} (m,p) = p - A^{-1} \mathcal{F} (m, p)$ for any $p \in G^2$ and we look for a fixed point for this function. Moreover, there holds $D_p \mathcal{G} (m, p) = I_4 - A^{-1} D_p \mathcal{F} (m, p)$. By applying Lemma \ref{lem:est_DF_A}, we get
    \begin{align*}
        \norm{D_p \mathcal{G} (m, p)} &\leq \norm{A^{-1}} \norm{D_p \mathcal{F} (m,p) - A} \\
            &\leq C \norm{A^{-1}} \bigl( q(y^+ - y^-) + \norm{m - {P}_{g^+, g^-}}_{H^1} \bigr).
    \end{align*}
    Therefore, if we take $p \in B_{p_0} (\xi)$ (ball in $G^2$ centered at $p_0$ of radius $\xi$, where $p_0 \coloneqq ((0,0), g_0)$) for some $\xi > 0$ to be defined later, and assuming $\xi < 1 < L_1'$, we get due to Lemma \ref{lem:diff_double_P_tilde}:
    \begin{align*}
        \norm{D_p \mathcal{G} (m, p)} &\leq C \norm{A^{-1}} \bigl( q (y_0) + \delta + \abs{g^+} + \abs{g^- - g_0} \bigr) \\
            &\leq C \norm{A^{-1}} \bigl( q(y_0) + \delta + \xi \bigr).
    \end{align*}
    Hence, $\norm{D_p \mathcal{G} (m, p)} \leq \frac{1}{2}$ as soon as
    \begin{equation} \label{eq:ass_mod_1}
        C \norm{A^{-1}} \Bigl( q(y_0) + \delta + \xi \Bigr) \leq \frac{1}{2}.
    \end{equation}
    On the other hand, we know that $\abs{\mathcal{F} (m, p_0)} \leq C \norm{m - P_{p_0}}_{H^1}$ thanks to Lemma \ref{prop:lin_est_F}.
    Thus,
    \begin{equation*}
        \abs{\mathcal{G} (m, p_0) - p_0} \leq \norm{A^{-1}} \abs{\mathcal{F} (m, p_0)} \leq C \norm{A^{-1}} \delta.
    \end{equation*}
    Moreover, by assuming \eqref{eq:ass_mod_1} so that $\norm{D_p \mathcal{G} (m, .)} \leq \frac{1}{2}$ on $B_{p_0} (\xi)$, we get for all $p \in B_{p_0} (\xi)$,
    \begin{equation*}
        \abs{\mathcal{G} (m, p) - \mathcal{G} (m, p_0)} \leq \frac{1}{2} \abs{p - p_0} \leq \frac{\xi}{2},
    \end{equation*}
    which yields
    \begin{equation*}
        \abs{\mathcal{G} (m, p) - p_0} \leq C \norm{A^{-1}} \delta + \frac{\xi}{2}.
    \end{equation*}
    This means that $\mathcal{G} (m, p) \in B_{p_0} (\xi)$ as soon as $C \norm{A^{-1}} \delta + \frac{\xi}{2} \leq \xi$, i.e.
    \begin{equation} \label{eq:ass_mod_2}
        C \norm{A^{-1}} \delta \leq \frac{\xi}{2}.
    \end{equation}
    From the previous computations, we conclude that $\mathcal{G} (m, \cdot)$ is a contraction on $B_{p_0} (\xi)$ as soon as \eqref{eq:ass_mod_1} and \eqref{eq:ass_mod_2} hold. 
    
    Let $\delta_1^\sharp >0$ and $L_1' \ge 1/\Gamma$ be so large that that 
    \[ C \|A^{-1}\| (q(L_1') + (1+2C \|A^{-1}\|) \delta_1^\sharp) \le \frac{1}{2}, \]
    and fix $\xi = 2 C \norm{A^{-1}} \delta_1^\sharp$. Then for $y_0  \ge - L_1'$ and $\delta \le \delta_1^\sharp$, \eqref{eq:ass_mod_1} and \eqref{eq:ass_mod_2} holds and $\mathcal{G}$ admits a unique fixed point in $B_{p_0} (2 C \norm{A^{-1}} \delta_1^\sharp)$. In particular, the conclusion of 1. follows.

    \medskip
    
    \emph{Step 2. Dependence on $g_0$}

    Let $\tilde \delta_1$ to be fixed later, and
     assume that there exists $g_0$ and $g_0'$ in $G$ such that $y_0,y_0' < - L_1'$ and $\norm{m - P_{0, g_0}}, \norm{m - P_{0, g_0'}} < \tilde \delta_1$. Then, there holds
    \begin{equation*}
        \norm{{P}_{0, g_0'} - {P}_{0, g_0}}_{H^1} \leq \norm{m - {P}_{0, g_0'}}_{H^1} + \norm{m - {P}_{0, g_0}}_{H^1} < 2 \tilde \delta_1.
    \end{equation*}
    On the other hand,
    \begin{equation*}
        {P}_{0, g_0'} - {P}_{0, g_0} = g_0' . w_*^- - g_0 . w_*^-,
    \end{equation*}
    therefore, as soon as $\delta_1'$ is small enough, we can apply Lemma \ref{lem:est_g_w} and get
    \begin{equation*}
        \abs{g_0' - g_0} \leq C \tilde \delta_1.
    \end{equation*}
    Now Step 1 provides $(\overline g^+,\overline g^-)$ and $(\overline g^+{}',\overline g^-{}')$ satisfying 1. and so
    \begin{equation*}
        \abs{\overline{g}^+} + \abs{\overline{g}^- - g_0}, \abs{\overline{g}^+{}'} + \abs{\overline{g}^-{}' - g_0'} \leq C \tilde \delta_1,
    \end{equation*}
    Therefore,
    \begin{equation*}
        \abs{\overline{g}^+{}'} + \abs{\overline{g}^-{}' - g_0} \leq C \tilde \delta_1.
    \end{equation*}
    Taking $\tilde \delta_1 > 0$ small enough, we get that $({g^+}', {g^-}') \in B_{p_0} ( 2 C \norm{A^{-1}} \delta_1^\sharp)$. By uniqueness of $(g^+, g^-)$ in this ball, we get $({g^+}', {g^-}') = (g^+, g^-)$.
    
    We therefore set $\delta_1' = \min(\delta_1^\sharp, \tilde \delta_1)$.
    \medskip
    
    \emph{Step 3. Regularity of the map}

    In Step 1, we only considered $p$ such that
    \begin{equation*}
        \norm{D_p \mathcal{F} (m,p) - A} \leq \frac{1}{2 \norm{A^{-1}}},
    \end{equation*}
    and thus $D_p \mathcal{F} (m,p)$ is invertible since
    \begin{equation*}
        A^{-1} D_p \mathcal{F} (m,p) = I_4 - (I_4 - A^{-1} D_p \mathcal{F} (m,p)),
    \end{equation*}
    with
    \begin{equation*}
        \norm{I_4 - A^{-1} D_p \mathcal{F} (m,p)}_{H^1} \leq \norm{A^{-1}} \norm{D_p \mathcal{F} (m,p) - A} \leq \frac{1}{2}.
    \end{equation*}
    This is in particular true for $(\overline{g}^+, \overline{g}^-)$. Therefore, the regularity of the application at $m$ can be deduced from the implicit function theorem applied on $\mathcal{F}$ (which is a $\mathscr{C}^\infty$ function since $\partial_x w_*^\iota$ and $e_1 \wedge w_*^+$ are $H^\infty$) at the point $(m, \overline{g}^+, \overline{g}^-)$.
\end{proof}

\subsection{Decomposition near a 2-domain wall under the \texorpdfstring{\eqref{eq:llg}}{(LLG)} flow}

We consider here the decomposition of a magnetisation defined on the interval $[0,T]$ for some $T>0$; the case of $[0,+\infty)$ being completely similar.

\begin{lem}[Continuous in time decomposition near a 2-domain wall] \label{lem:decomp_time}
    There exist $\delta_1 > 0$ and $L_1 > 0$ such that the following holds. Let $L \ge L_1$, $T > 0$ and $m \in \mathscr{C} ([0,T], \mathcal{H}^1)$ satisfying, for some $y_0 > L_1$,
    \begin{equation*}
        \delta \coloneqq \sup_{t \in [0,T]}  \inf_{\substack{g^+ \in G_{>L} \\ g^- \in G_{<-L}}}
        \norm{m(t) - P_{g^+, g^-}}_{H^1} < \delta_1
    \end{equation*}
    Then  there exists $\overline{g}^\iota = (\overline{y}^\iota, \overline{\phi}^\iota) \in \mathscr{C}( [0,T], G)$ for $\iota = \pm$ and $\varepsilon \in \mathscr{C} ([0,T], H^{1})$ such that, for all $t \in [0,T]$,
    \begin{itemize}
        \item $\iota \overline{y}^\iota (t) \geq L-1$,
        \item $m(t) = \overline{g}^+ (t) . w_*^+ + \overline{g}^- (t) . w_*^- + e_1 + \varepsilon (t)$,
        \item $\mathcal{F} (m(t), \overline{g}^+, \overline{g}^-) = \overline{\mathcal{F}} (\varepsilon (t), \overline{g}^+, \overline{g}^-) = 0$,
        \item $\norm{\varepsilon (t)}_{H^1} \leq C \Bigl( \delta + q (\overline y^+ - \overline y^-) \Bigr)$.
    \end{itemize}
    Moreover, if $m$ is $\mathscr{C}^1 ([0,T], \mathcal{H}^1)$ (resp. $W^{1, \infty}_\textnormal{loc} ([0,T], \mathcal{H}^1)$), then both $\overline{g}^\iota$ are $\mathscr{C}^1 ([0,T])$ (resp. locally Lipschitz).
\end{lem}

As before, a similar statement with $T=+\infty$ with the obvious modifications holds.

\begin{proof}
    As $m \in \mathscr{C} ([0,T], \mathcal{H}^1)$, $m$ is uniformly continuous on $[0,T]$ (in the case $[0,+\infty)$, one argues on any compact subinterval): We can thus find $0 = t_0 < t_1 < \dots < t_N = T$ such that, for all $0 \leq 0 \leq N - 1$ and $t \in [t_k, t_{k+1}]$, there holds $\norm{m (t) - m(t_k)}_{H^1} \leq \delta$. Then, for any $k$, we can find ${g}^\iota_k \in G$ ($\iota = \pm 1$) such that ${y}^\iota_k \geq y_0$ and
    \begin{equation*}
        \norm{m(t) - P_{{g}_k^+, {g}_k^-}}_{H^1} < 2 \delta.
    \end{equation*}
    In particular, for all $k$,
    \begin{equation*}
        \norm{P_{{g}_{k+1}^+, {g}_{k+1}^-} - P_{{g}_k^+, {g}_k^-}}_{H^1} < 5 \delta.
    \end{equation*}
    From Lemma \ref{lem:diff_double_P_tilde}, we thus get
    \begin{equation*}
        \norm{{g}_{k+1}^+ . w_*^+ - {g}_k^+ . w_*^+}_{H^1}^2 + \norm{{g}_{k+1}^- . w_*^- - {g}_k^- . w_*^-}_{H^1}^2 \leq 25 \delta^2 + C q(2y_0).
    \end{equation*}
    Therefore, if we assume that $\delta$ is small enough and $y_0$ large enough, we can apply \cite[Claim~4.12]{Cote_Ignat__stab_DW_LLG_DM} and get integers $n^+_{k+1}$ and $n^-_{k+1}$ such that, by noting $\mu = \delta + \sqrt{q(2y_0)}$,
    \begin{equation*}
        \abs{{g}_{k+1}^\iota + (0, 2 \pi n_{k+1}^\iota) - {g}_k^\iota} \leq C \norm{{g}^\iota_{k+1} . w_*^\iota - {g}^\iota_k . w_*^\iota}_{H^1} \leq C \mu.
    \end{equation*}
    We can then change every ${\phi}_k^\iota$ by adding $2 \pi {n_k^\iota}'$ for some well chosen ${n_k^\iota}' \in \mathbb{Z}$ such that, for all $k$,
    \begin{equation*}
        \abs{{g}_{k+1}^\iota - {g}_k^\iota} \leq C \mu.
    \end{equation*}
    Consider now $\tilde{g}^\iota_1$ affine on each segment $[t_k, t_{k+1}]$ such that $\tilde{g}^\iota_1 (t_k) = {g}^\iota_k$ for $k = 0, \dots, N$, and then consider smooth functions $\tilde{g}^\iota_0$ such that $\norm{\tilde{g}^\iota_0 - \tilde{g}^\iota_1}_{\mathscr{C} ([0,T])} \leq \mu$. Thus we get for any $t \in [t_k, t_{k+1}]$ by applying Lemma \ref{lem:est_g_w}
    \begin{equation*}
        \norm{\tilde{g}^\iota_1 (t) . w_*^\iota - \tilde{g}^\iota_1 (t_k) . w_*^\iota}_{H^1} \leq C \abs{\tilde{g}^\iota_1 (t) - \tilde{g}^\iota_1 (t_k)} \leq C \abs{\tilde{g}^\iota_1 (t_{k+1}) - \tilde{g}^\iota_1 (t_k)} \leq C \mu,
    \end{equation*}
    and therefore, by using Lemma \ref{lem:diff_double_P_tilde},
    \begin{align*}
        \norm{m (t) - P_{\tilde{g}^+_0 (t), \tilde{g}^-_0 (t)}}_{H^1} &\leq
            \begin{multlined}[t]
                \norm{m(t) - m(t_k)}_{H^1} + \norm{m(t_k) - P_{\tilde{g}^+_1 (t_k), \tilde{g}^-_1 (t_k)}}_{H^1} \\ + \norm{P_{\tilde{g}^+_1 (t_k), \tilde{g}^-_1 (t_k)} - P_{\tilde{g}^+_1 (t), \tilde{g}^-_1 (t)}}_{H^1} + \norm{P_{\tilde{g}^+_1 (t), \tilde{g}^-_1 (t)} - P_{\tilde{g}^+_0 (t), \tilde{g}^-_0 (t)}}_{H^1}
            \end{multlined} \\
            &\leq \delta + 2 \delta + C \mu + C \mu \leq C \mu.
    \end{align*}
    By assuming $\mu$ small enough, i.e. $\delta$ small enough and $y_0$ large enough, we can assume
    \begin{equation*}
        \norm{(- \tilde{g}^+_0 (t)).m (t) - P_{0, (\tilde{g}^+_0)' (t) + \tilde{g}^-_0 (t)}}_{H^1} = \norm{m (t) - P_{\tilde{g}^+_0 (t), \tilde{g}^-_0 (t)}}_{H^1} \leq \delta_1'
    \end{equation*}
    where $\delta_1'$ is given in Lemma \ref{lem:static_implicit} and $\tilde{g}^+_0 = (\tilde{y}^+_0, \tilde{\phi}^+_0)$ gives $(\tilde{g}^+_0)' = (\tilde{y}^+_0, - \tilde{\phi}^+_0)$. Moreover, for any $t \in [t_k, t_{k+1}]$, we have $\abs{\tilde{g}^\iota_0 (t) - g_k^\iota} \leq C \mu$, which gives
    \begin{equation*}
        \pm \tilde{y}^\pm_0 (t) \geq \pm y_k^\pm - C \mu \geq y_0 - C \mu \qquad
        \text{and} \qquad
        \tilde{y}^+_0 (t) - \tilde{y}^-_0 (t) \geq 2 y_0 - C \mu \geq L_1',
    \end{equation*}
    for $\mu$ small enough and $y_0$ large enough ($L_1'$ is given in Lemma \ref{lem:static_implicit}). . These conditions define $\delta_1>0$ small and $L_1 \ge 1$ large.
    Then, we can apply Lemma \ref{lem:static_implicit} to $(- \tilde{g}^+_0 (t)).m (t)$, which gives some $\overline{g}^+_0 (t), \overline{g}^-_0 (t) \in G$ for any $t \in I$ such that
    \begin{itemize}
        \item $\abs{\overline{g}^-_0 (t) - \Bigl( (\tilde{g}^+_0)' (t) + \tilde{g}^-_0 (t) \Bigr)} \leq C \mu$ and $\abs{\overline{g}^+_0 (t)} \leq C \mu$,
        \item $\norm{(- \tilde{g}^+_0 (t)).m (t) - P_{\overline{g}^+_0, \overline{g}^-_0}}_{H^1} \leq C \mu$,
        \item $\mathcal{F} ((- \tilde{g}^+_0 (t)).m (t), \overline{g}^+_0, \overline{g}^-_0) = 0$.
    \end{itemize}
    Then, define
    \begin{equation*}
        \overline{g}^+ \coloneqq \tilde{g}^+_0 + \overline{g}^+_0, \qquad
        \overline{g}^- \coloneqq \overline{g}^-_0 - (\tilde{g}^+_0)'.
    \end{equation*}
    These gauges satisfy $\mathcal{F} (m (t), \overline{g}^+ (t), \overline{g}^- (t)) = 0$, and also $\abs{\overline{g}^\iota (t) - \tilde{g}^\iota_0 (t)} \leq C \mu$, which means that
    \begin{equation*}
        \iota \overline{y}^\iota \geq \tilde{y}_0^\iota - C \mu \geq y_0 - C \mu \geq y_0 - 1,
    \end{equation*}
    (for $\mu $ small enough, up to modifying the choices of $\delta_1$ and $L_1$) and where $\overline{g}^\iota \eqqcolon (\overline{y}^\iota, \overline{\phi}^\iota)$. Moreover, since $\tilde{g}^+_0$ is smooth, $(- \tilde{g}^+_0).m$ has the same regularity as $m$. Therefore, if $m$ is $\mathscr{C}^1 ([0,T], \mathcal{H}^1)$, then the regularity result of Lemma \ref{lem:static_implicit} gives that $\overline{g}^\iota_0 \in \mathscr{C}^1 ([0,T])$ and so are $\overline{g}^\iota$. Similar arguments when $m \in W^{1, \infty} ([0,T], \mathcal{H}^1)$ give the conclusion.
\end{proof}

\begin{lem}[Evolution equations of the gauges] \label{lem:64}
    Up to further reducing $\delta_1>0$ and increasing $L_1 \ge 1$, if $m \in \mathscr{C} ([0,T], \mathcal{H}^2)$ is a solution of \eqref{eq:llg}, then both $\overline{g}^\iota$ given by Lemma \ref{lem:decomp_time} are Lipschitz and satisfy, for a.e. $t \in [0,T]$,
    \begin{equation*}
        \abs{\dot{\overline{g}}^\iota (t) - \dot{g}_*^\iota (t)} \leq C \bigl( \norm{\varepsilon (t)}_{H^1} + q(\overline{y}^+ (t) - \overline{y}^- (t)) \bigr),
    \end{equation*}
\end{lem}

\begin{proof}
    Let assume first that $m \in \mathscr{C} ([0,T], \mathcal{H}^3)$. From \eqref{eq:llg}, we get $\partial_t m \in \mathscr{C} ([0,T], H^1)$. Therefore, both $\overline{g}^\iota$ given by Lemma \ref{lem:decomp_time} are $\mathscr{C}^1 ([0,T])$. Then, we can compute using the fact that $\varepsilon \in \mathscr{C}^1 ([0,T], H^1)$:
    \begin{equation*}
        \partial_t m = \partial_t \varepsilon - \dot{\overline{y}^+} \overline{g}^+ . \partial_x w_*^+ + \dot{\overline{\phi}^+} e_1 \wedge \overline{g}^+ . w_*^+ - \dot{\overline{y}^-} \overline{g}^- . \partial_x w_*^- + \dot{\overline{\phi}^-} e_1 \wedge \overline{g}^- . w_*^-
    \end{equation*}
    Then, $\delta E$ is linear and $\delta E (w_*^\pm) = \beta_* w_*^\pm$, so
    \begin{equation*}
        H(m) = - \overline{g}^+ . (\beta_* w^+) - \overline{g}^- . (\beta_* w_*^-) - \delta E (\varepsilon) + h (t) e_1,
    \end{equation*}
    with $\delta E (\varepsilon) = O_2^1 (\varepsilon)$.
    Then, we also have
    \begin{align*}
        m \wedge H(m) &=
            \begin{multlined}[t]
                h(t) \Bigl[ \overline{g}^+ . w_*^+ \wedge e_1 + \overline{g}^- . w_*^- \wedge e_1 \Bigr] \\
                \begin{aligned}
                    &- \Bigl[ \Bigl( \overline{g}^- . w_*^- + e_1 \Bigr) \wedge \overline{g}^+ . (\beta_* w_*^+) + (\overline{g}^+ . w_*^+ + e_1) \wedge \Bigl( g^- . (\beta_* w_*^-) \Bigr) \Bigr] \\
                    &+ h(t) \varepsilon \wedge e_1 - \varepsilon \wedge \Bigl[ \overline{g}^+ . (\beta_* w_*^+) + \overline{g}^- . (\beta_* w_*^-) \Bigr]
                    - P_{\overline{g}^+, \overline{g}^-} \wedge \delta E (\varepsilon) - \varepsilon \wedge \delta E (\varepsilon)
                \end{aligned}
            \end{multlined} \\
            &= h(t) \Bigl[ \overline{g}^+ . w_*^+ \wedge e_1 + \overline{g}^- . w_*^- \wedge e_1 \Bigr] + O (f_{\overline{y}^+, \overline{y}^-} (x)) + O_2^1 (\varepsilon) + O_2^2 (\varepsilon),
    \end{align*}
    by using Corollary \ref{cor:est_w_orthogonal} and where
    \begin{equation*}
        f_{\overline{y}^+, \overline{y}^-} (x) = e^{- \Gamma (y^+ - y^-) } \times
        \begin{cases}
            e^{-\Gamma (x-y^+)} &\qquad \text{if } x \geq y^+ \\
            1 &\qquad \text{if } y \in [y^-, y^+] \\
            e^{\Gamma (x + y^-)} &\qquad \text{if } x \leq y^-
        \end{cases}.
    \end{equation*}
    Last, there also holds
\[
        m \wedge (m \wedge H(m)) = h(t) \Bigl[ \overline{g}^+ . (w_*^+ \wedge (w_*^+ \wedge e_1) + g^- . (w_*^- \wedge (w_*^- \wedge e_1)) \Bigr]+ O ( f_{\overline{y}^+, \overline{y}^-} (x) ) + O_2^1 (\varepsilon) + O_2^2 (\varepsilon) + O_2^3 (\varepsilon).
\]   
    Hence,
    \begin{align}
        \partial_t \varepsilon & = \dot{\overline{y}^+} \overline{g}^+ . \partial_x w_*^+ - \dot{\overline{\phi}^+} (e_1 \wedge \overline{g}^+ . w_*^+) + \dot{\overline{y}^-} (\overline{g}^- . \partial_x w_*^-) - \dot{\overline{\phi}^-} e_1 \wedge (\overline{g}^- . w_*^-)  \nonumber \\
                & \qquad + h(t) \Bigl[ \overline{g}^+ . w_*^+ \wedge e_1 + \overline{g}^- . w_*^- \wedge e_1 \Bigr] - \alpha h(t) \Bigl[ \overline{g}^+ . (w_*^+ \wedge (w_*^+ \wedge e_1)) + \overline{g}^- . (w_*^- \wedge (w_*^- \wedge e_1)) \Bigr] \nonumber \\
                & \qquad + O ( f_{\overline{y}^+, \overline{y}^-} (x) ) + O_2^1 (\varepsilon) + O_2^2 (\varepsilon) + O_2^3 (\varepsilon).  \label{eq:dt_eps}
    \end{align}
    From this equation, we can derive the first order of $\dot{\overline{g}}^\iota$. First, recall that $\overline{\mathcal{F}} (\varepsilon, \overline{g}^+, \overline{g}^-) = 0$. By differentiating with respect to $t$, we thus get for example
    \begin{align*}
        \int \partial_t \varepsilon \cdot (g^+ . \partial_x w_*^+) \diff x &= - \int \varepsilon \cdot \partial_t (g^+ . \partial_x w_*^+) \diff x \\
            &= \dot{\overline{y}}^+ \int \varepsilon \cdot (g^+ . \partial^2_{xx} w_*^+) \diff x - \dot{\overline{\phi}}^+ \int \varepsilon \cdot (e_1 \wedge g^+ . \partial_x w_*^+) \diff x \eqqcolon M_1 (\varepsilon)
            \begin{pmatrix}
                \dot{\overline{y}}^+ \\
                \dot{\overline{\phi}}^+
            \end{pmatrix},
    \end{align*}
    where the $2 \times 2$ matrix $M_1 (\varepsilon)$ satisfies $\norm{M_1 (\varepsilon)} \leq C \norm{\varepsilon}_{H^1}$. More generally, we obtain
    \begin{equation} \label{eq:F_dt_eps_1}
        \overline{\mathcal{F}} (\partial_t \varepsilon, \overline{g}^+, \overline{g}^-) = M_0 (\varepsilon) 
            \begin{pmatrix}
                \dot{\overline{y}}^+ \\
                \dot{\overline{\phi}}^+ \\
                \dot{\overline{y}}^- \\
                \dot{\overline{\phi}}^-
            \end{pmatrix},
    \end{equation}
    with a $4 \times 4$ matrix $M_0 (\varepsilon)$ such that $\norm{M_0 (\varepsilon)} \leq C \norm{\varepsilon}_{H^1}$.
    On the other hand, we can also compute $\overline{\mathcal{F}} (\partial_t \varepsilon, \overline{g}^+, \overline{g}^-)$ with the relation \eqref{eq:dt_eps}, and use Lemma \ref{lem:int_w} for the zeroth order terms, Corollary \ref{cor:est_w_orthogonal} for the terms involving both $w_*^+$ and $w_*^-$, Proposition \ref{prop:lin_est_F} for the terms in $O ( f_{\overline{y}^+, \overline{y}^-} (x) )$ and \cite[Claim~4.9]{Cote_Ignat__stab_DW_LLG_DM} for $O_2^\ell (\varepsilon)$. For example, one of the terms involving both $w_*^+$ and $w_*^-$ is
    \begin{equation*}
        \int (\overline{g}^- . \partial_x w_*^-) \cdot (\overline{g}^+ . \partial_x w_*^+) \diff x,
    \end{equation*}
    and from Corollary \ref{cor:est_w_orthogonal2}, we can estimate
    \begin{equation*}
        \abs{\int (\overline{g}^- . \partial_x w_*^-) \cdot (\overline{g}^+ . \partial_x w_*^+) \diff x} \leq C q(\overline{y}^+ - \overline{y}^-).
    \end{equation*}
    From this, we obtain
    \begin{equation} \label{eq:F_dt_eps_2}
        \overline{\mathcal{F}} (\partial_t \varepsilon, \overline{g}^+, \overline{g}^-) = \frac{2}{\Gamma} (B + B_0) 
            \begin{pmatrix}
                \dot{\overline{y}}^+ \\
                \dot{\overline{\phi}}^+ \\
                \dot{\overline{y}}^- \\
                \dot{\overline{\phi}}^-
            \end{pmatrix}
            + \frac{2 h(t)}{\Gamma}
            \begin{pmatrix}
                - \gamma + \alpha \Gamma \\
                1 \\
                - \gamma - \alpha \Gamma \\
                1
            \end{pmatrix}
            + O \bigl( q(\overline{y}^+ - \overline{y}^-) + \norm{\varepsilon}_{H^1} + \norm{\varepsilon}_{H^1}^3 \bigr),
    \end{equation}
    where
    \begin{equation*}
        B = 
        \begin{pmatrix}
            1 & \gamma & 0 & 0 \\
            - \gamma & - 1 & 0 & 0 \\
            0 & 0 & 1 & \gamma \\
            0 & 0 & - \gamma & - 1
        \end{pmatrix}
        =
        \begin{pmatrix}
            B_\gamma & 0_2 \\
            0_2 & B_{\gamma}
        \end{pmatrix}, \qquad
        B_{\gamma} = 
            \begin{pmatrix}
                1 & \gamma \\
                - \gamma & - 1 
            \end{pmatrix}
    \end{equation*}
    and $B_0$ is a $4 \times 4$ matrix which satisfies $\norm{B_0} \leq C q(\overline{y}^+ - \overline{y}^-)$. Hence, \eqref{eq:F_dt_eps_1} and \eqref{eq:F_dt_eps_2} imply
    \begin{equation*}
        (B + \tilde{B}_0) 
            \begin{pmatrix}
                \dot{\overline{y}}^+ \\
                \dot{\overline{\phi}}^+ \\
                \dot{\overline{y}}^- \\
                \dot{\overline{\phi}}^-
            \end{pmatrix}
            = - h(t)
            \begin{pmatrix}
                - \gamma + \alpha \Gamma \\
                1 \\
                - \gamma - \alpha \Gamma \\
                1
            \end{pmatrix}
            + O \bigl( q(\overline{y}^+ - \overline{y}^-)  + \norm{\varepsilon}_{H^1} + \norm{\varepsilon}_{H^1}^3 \bigr),
    \end{equation*}
    with
    \begin{equation*}
        \norm{\tilde B_0} \leq C \Bigl( q(\overline{y}^+ - \overline{y}^-) + \norm{\varepsilon}_{H^1} \Bigr) \leq C \bigl( q(2y_0) + \delta \bigr).
    \end{equation*}
    We know that $B_{\gamma}$ is invertible because $\gamma^2 < 1$, with inverse $\Gamma^{- 2} B_{\gamma}$, and thus $B$ is also invertible. Therefore, as soon as $q(2y_0) + \delta$ is small enough, $B + \tilde B_0$ is invertible with inverse
    \begin{equation*}
        (B + \tilde B_0)^{-1} = \frac{1}{\Gamma^2}
        \begin{pmatrix}
            B_{\gamma} & 0_2 \\
            0_2 & B_{- \gamma}
        \end{pmatrix}
        + O \bigl( q(\overline{y}^+ - \overline{y}^-) + \norm{\varepsilon}_{H^1} \bigr),
    \end{equation*}
    which leads to
    \begin{equation*}
        (\dot{\overline{g}}^+, \dot{\overline{g}}^-) - (\dot{g}_*^+, \dot{g}_*^-) = O \bigl( q(\overline{y}^+ - \overline{y}^-) + \norm{\varepsilon}_{H^1} + \norm{\varepsilon}_{H^1}^3 \bigr),
    \end{equation*}
    hence the conclusion for the $\mathcal{H}^3$ case. We point out that this estimate only depends on the $H^1$ norm of $\varepsilon$. Therefore, for the general case $m \in \mathscr{C} ([0,T], \mathcal{H}^2)$, we can use a limiting argument: we refer to the proof of Proposition 4.11 (step 4) in \cite{Cote_Ignat__stab_DW_LLG_DM} for further details.
\end{proof}

Lemma \ref{lem:decomp_magn} gathers the content of the above Lemmas \ref{lem:decomp_time} and \ref{lem:64}.

\section*{Acknowledgements}

The authors acknowledge partial support by the ANR project MOSICOF ANR-21-CE40-0004. The research of R.C. has benefitted from support provided by the University of Strasbourg Institute for Advanced Study (USIAS) for a Fellowship within the French national programme ``Investment for the future'' (IdEx-Unistra). The authors would like to thank Clémentine Courtès and Yannick Privat for several discussions during the preparation of this article.

\printindex

\bibliographystyle{abbrv}
\bibliography{sample}

\end{document}